\documentclass[a4paper,oneside]{amsart}

\usepackage[utf8]{inputenc} 
\usepackage[T1]{fontenc}
\usepackage[english]{babel}
\usepackage{amssymb}
\usepackage{graphicx}
\usepackage{listings}
\usepackage{mathtools}
\usepackage{mathrsfs}
\usepackage{color}
\usepackage{geometry}
\usepackage{titlesec}
\usepackage{eufrak}
\usepackage{hyperref}

\newcommand{\RR}{\mathbb{R}}
\newcommand{\CC}{\mathbb{C}}
\newcommand{\NN}{\mathbb{N}}

\renewcommand{\ge}{\geqslant}
\renewcommand{\le}{\leqslant}

\newcommand{\R}{\mathrm{Re}}
\newcommand{\I}{\mathrm{Im}}

\newtheorem{Th}{Theorem}[section]
\newtheorem{Prop}[Th]{Proposition}
\newtheorem{Lem}[Th]{Lemma}
\theoremstyle{definition}
\newtheorem{definition}[Th]{Definition}

\newtheorem{Claim}[Th]{Claim}

\theoremstyle{remark}
\newtheorem{Rq}[Th]{Remark}

\titleformat{\section}[block]
{\centering\scshape\Large\bfseries}
{\normalfont\textbf{\thesection}.}{0.5em}{}

\titleformat{\subsection}[runin]
{\normalfont\bfseries}
{\thesubsection.}{0.5em}{}[.]
\titlespacing{\subsection}{1pc}{1.5ex plus .1ex minus .2ex}{0.5pc}

\titleformat{\subsubsection}[runin]
{\normalfont\bfseries}
{\thesubsubsection.}{0.5em}{}[.]
\titlespacing{\subsubsection}{1pc}{1.5ex plus .1ex minus .2ex}{0.5pc}
\numberwithin{equation}{section}

\titleformat{\paragraph}[runin]
{\normalfont\itshape}
{\theparagraph.}{0.5em}{}[.]
\titlespacing{\paragraph}{1pc}{1.5ex plus .1ex minus .2ex}{0.5pc}

\newcounter{Hyp}

\newcounter{Hypo}

\title[Spatial decay of multi-solitons of the gKdV and NLS equations]{Spatial decay of multi-solitons of the generalized Korteweg-de Vries and nonlinear Schrödinger equations}

\author{Raphaël Côte}
\address{Institut de Recherche Mathématique Avancée UMR 7501, Université de Strasbourg, Strasbourg, France}
\email{cote@math.unistra.fr}

\author{Xavier Friederich}
\address{Institut de Recherche Mathématique Avancée UMR 7501, Université de Strasbourg, Strasbourg, France}
\email{friederich@math.unistra.fr}


\begin{document}

\begin{abstract}
We study pointwise spatial decay of multi-solitons of the generalized Korteweg-de Vries equations. We obtain that, uniformly in time, these solutions and their derivatives decay exponentially in space on the left of and in the solitons region, and prove rapid decay on the right of the solitons. 

We also prove the corresponding result for multi-solitons of the nonlinear Schrödinger equations, that is, exponential decay in the solitons region and rapid decay outside.
\end{abstract}

\maketitle

\section{Introduction}

\subsection{(gKdV) multi-solitons}

We consider the generalized Kor\-te\-weg-de Vries equations 
\begin{gather}\tag{gKdV}\label{gKdV}
\partial_t u+\partial_x(\partial^2_xu+u^p)=0
\end{gather}
where $(t,x) \in \RR\times\RR$ and $p\ge 2$ is an integer. 

Recall that \eqref{gKdV} admits a family of explicit traveling wave solutions $R_{c_0,x_0}$ indexed by $(c_0,x_0) \in \RR^*_+\times\RR$. Let $Q$ be the unique (up to translation) positive solution in $H^1(\RR)$ (known also as \textit{ground state}) to the following stationary elliptic problem associated with (gKdV)
\[ Q''+Q^p=Q, \]
given by the explicit formula
\[ Q(x)=\left(\frac{p+1}{2\cosh^2\left(\frac{p-1}{2}x\right)}\right)^{\frac{1}{p-1}}. \]
Then for all $c_0>0$ (velocity parameter) and $x_0\in\RR$ (translation parameter), 
\begin{equation}\label{soliton1} 
R_{c_0,x_0}(t,x)=Q_{c_0}(x-c_0t-x_0)
\end{equation} 
is a global traveling wave solution of (gKdV) classically called the \textit{soliton} solution, where $Q_{c_0}(x)=c_0^{\frac{1}{p-1}}Q(\sqrt{c_0}x)$.

We are interested here in qualitative properties of the \emph{multi-solitons}, which are solutions to \eqref{gKdV} built upon solitons, and are defined as follows.

\begin{definition} \label{def:multi_sol}
Let $N\ge 1$ and consider $N$ solitons $R_{c_j,x_j}$ as in \eqref{soliton1} with speeds $0<c_1<\cdots<c_N$. A \textit{multi-soliton} \textit{in $+\infty$} (resp. \textit{in $-\infty$}) associated with the $R_{c_j,x_j}$ is an $H^1$-solution $u$ of (gKdV) defined in a neighborhood of $+\infty$ (resp. $-\infty$) and such that \begin{equation}\label{def_multisol}
\left\|u(t)-\sum_{j=1}^NR_{c_j,x_j}(t)\right\|_{H^1}\rightarrow 0,\qquad \text{as } t\to +\infty \text{ (resp. as } t\to-\infty\text{)}.
\end{equation}
\end{definition}

The study of multi-solitons is motivated by the soliton resolution conjecture, which asserts that generic solutions to nonlinear dispersive equations should behave as a sum of decoupled solitons for large times. Such a resolution was obtained for the original Korteweg-de Vries equation (KdV), corresponding to $p=2$; and the modified Korteweg-de Vries equations (mKdV), corresponding to $p=3$. We refer to \cite{eckhaus,schuur,chen} for instance.

In the context of \eqref{gKdV}, multi-solitons were first constructed for (KdV) and (mKdV), via the inverse scattering transform. It provides explicit formulas: for (KdV), it writes
\begin{equation} \label{eq:Nsol_kdv} u=6\frac{\partial^2}{\partial x^2}\ln\det M,
\end{equation}
where $M(t,x)$ is the $N\times N$-matrix with generic entry
\[ M_{(i,j)}(t,x)=\delta_{i,j}+2\frac{(c_ic_j)^{\frac{1}{4}}}{\sqrt{c_i}+\sqrt{c_j}}e^{\frac{1}{2}\left(\sqrt{c_i}(x-c_it)+x_i+\sqrt{c_j}(x-c_jt)+x_j\right)} \]
and $\delta_{i,j}$ stands for the Kronecker symbol: see \cite{ggkm}, \cite[section 6]{miura}, or \cite{hirota}. We refer to Schuur \cite[chapter 5, (5.5)]{schuur} and to Lamb \cite[chapter 5]{lamb} for a formula for (mKdV). One can observe from these formulas that such a solution $u$ in \eqref{eq:Nsol_kdv} is a multi-soliton both in $\pm\infty$ \cite{miura}, with the same velocity parameters $c_i$ in $\pm\infty$,  but with distinct translation parameters whose shifts can be quantified in terms of the $c_i$.

\bigskip

The construction of multi-solitons was subsequently extended to many non integrable models, and \eqref{gKdV} is probably the equation for which their study has been most developed. One important result concerns the complete classification of the multi-solitons, depending on the value of $p$ with respect to $5$: recall that for $p<5$, \eqref{gKdV} is $L^2$-subcritical and solitons are stable, $p=5$ is the $L^2$ critical equation, and for $p>5$, \eqref{gKdV} is $L^2$-supercritical and solitons are unstable, linearly and nonlinearly. Let us recall it below.

\begin{Th}[Martel \cite{martel}; Côte, Martel and Merle \cite{cmm}; Combet \cite{combetgkdv}]\label{th_multi_known}
Let $p>1$ be an integer and let $N\ge 1$, $0<c_1<\dots<c_N$, and $x_1,\dots, x_N\in\RR$.

If $p\le 5$, there exists $T_0\ge 0$ and a unique multi-soliton $u\in\mathscr{C}([T_0,+\infty),H^1(\RR))$  (in $+\infty$) associated with the $R_{c_i,x_i}$, $i\in\{1,\dots,N\}$.

If $p>5$, there exists a one-to-one map $\Phi$ from $\RR^N$ to the set of all $H^1$-solutions of \eqref{gKdV} defined in a neighborhood of $+\infty$ such that $u$ is a multi-soliton in $+\infty$ associated with the $R_{c_i,x_i}$ if and only if there exist $\lambda\in\RR^N$ and $T_0\ge 0$ such that $u_{|[T_0,+\infty)}=\Phi(\lambda)_{|[T_0,+\infty)}$.

Moreover, in each case, $u$ belongs to $\mathscr{C}([T_0,+\infty),H^s(\RR))$ for all $s\ge 0$, and there exist $\theta>0$ (depending on the $c_i$ but independent of $s$) and positive constants $\lambda_s$ such that for all $s\ge 0$,
\begin{equation}\label{CVexp_Hs}
\forall  t \ge T_0, \quad \left\|u(t)-\sum_{j=1}^NR_{c_j,x_j}(t)\right\|_{H^s}\le \lambda_s e^{-\theta t}.
\end{equation}
\end{Th}

\begin{Rq}
One can take 
\begin{equation} \label{def:theta}
\theta:=\frac{1}{32}\min\left\{c_1,\min_{j=1,\dots,N-1}\left\{c_{j+1}-c_j\right\}\right\}^{\frac{3}{2}}
\end{equation}
in the above theorem.
\end{Rq}

\bigskip

In this article, we go on studying qualitative properties of the multi-solitons and we are concerned with the behavior in space (at fixed time) of the multi-solitons of \eqref{gKdV}. From formula \eqref{eq:Nsol_kdv}, one can show that (KdV) multi-solitons are exponentially localized away from  the centers of the involved solitons (as solitons are); this is also the case for (mKdV) multi-solitons. Our goal is to extend this property to multi-solitons of \eqref{gKdV}, that is, non integrable equations.

Our main result is that there is indeed exponential decay in the solitons region and on the left of the train of solitons, and rapid decay on the right of it. Let us state this more precisely.

\begin{Th}\label{th_decay}
Fix the parameters $0<c_1<\dots<c_N$ and $x_1,\dots,x_N\in\RR$, and let  $u$ be a multi-soliton of \eqref{gKdV} associated with the solitons $R_{c_j,x_j}$, as in Theorem \ref{th_multi_known} ($u$ is unique if $p\le 5$). 

Let $\beta>c_N$.  Then there exist $T_1>0$ and $\kappa>0$  such that for all $s\in\NN$, 
\begin{enumerate}
\item \emph{(Exponential decay in the solitons region and to its left)}
 there exists $C_s>0$ such that for all $t\ge T_1$,
\begin{equation}\label{dec_exp_mid}
\forall x \le \beta t, \quad \left|\partial^s_x u(t,x)\right|\le C_{s}\sum_{j=1}^Ne^{-\kappa |x-c_jt|};
\end{equation}
\item \emph{(Algebraic decay to the right of the last soliton)}
for all $n\in\NN$, there exists $C_{s,n}>0$ such that for all $t\ge T_1$,
\begin{equation}\label{dec_right}
\forall x > \beta t, \quad \left|\partial^s_x u(t,x)\right|\le \frac{C_{s,n}}{(x-\beta t)^n}.
\end{equation}
\end{enumerate}
\end{Th}

\begin{Rq}
Our bounds give $\kappa = O(1/\beta)$ as $\beta \to +\infty$, where the implicit constant depends on the $c_i$; precise rates are stated in Proposition \ref{prop:left} and \ref{prop:soliton}.
\end{Rq}

\begin{Rq}
Theorem \ref{th_decay} can be extended to more general nonlinearities, of the form
\[ \partial_t u + \partial_{x} ( \partial_{xx} u + f(u)) =0 \]
where $f: \RR \to \RR$ is $\mathscr C^\infty$, convex on $\RR^+$ and $f(0) = f'(0)=0$, and the velocities $c_j$ are such that there is a $\mathcal C^1$ map $\mathcal V_j \to \mathscr{S} (\RR)$, $c \mapsto Q_c$ defined on some neighborhood $\mathcal V_j$ of $c_j$, such that $\partial_{xx} Q_c + f(Q_c) = c Q_c$ and 
\begin{equation} \label{eq:noncrit}
\left. \frac{d}{dc} \int Q_c^2 \; dx \right|_{c=c_j} \ne 0.
\end{equation}
These conditions ensure the existence of multi-solitons as done in \cite[Theorem 3]{cmm} (as noted there, condition \eqref{eq:noncrit} can probably be avoided).
\end{Rq}

\subsection{(NLS) multi-solitons}

We extend the decay properties in Theorem \ref{th_decay} to the $d$-dimensional nonlinear Schrödinger equation
\begin{gather}\tag{NLS}\label{NLS}
i\partial_t u+\Delta u+|u|^{p-1}u=0
\end{gather}
where $(t,x)$ is taken in $\RR\times\RR^d$ ($d\ge 1$), $u$ is a complex valued function, and the nonlinearity is $H^1$-subcritical, that is $1<p<1+\frac{4}{(d-2)_+}$.

Due to the $H^1$-subcritical assumption, for all $\omega>0$, there exists a unique positive radial solution $Q_\omega\in H^1(\RR^d)$ to the stationary equation
\[ \Delta Q_\omega+Q_\omega^p=\omega Q_\omega. \]
Moreover, if $z\mapsto |z|^{p-1}z$ is $\mathscr{C}^s$ on $\CC$ (as an $\RR$-differentiable function), then $Q_{\omega}$ is $\mathscr{C}^{s+2}$ on $\RR^d$ and, for all $s'\in\{0,\dots,s+2\}$, there exists a constant $C_{s'}$ depending on $s'$ such that for each multi-index $\sigma\in\NN^d$ with $|\sigma|=s'$ the following exponential decay property holds \cite{berestycki}\footnote{We denote $|x|$ the canonical euclidian norm of $x \in \RR^d$ (or the modulus in $\CC$), and $|\sigma|$ the length of the multi-index $\sigma \in \NN^d$; the context makes it unambiguous.}

\begin{equation}\label{exp_decay_ground_state}
\forall x\in\RR^d,\qquad \left|\partial^\sigma Q_{\omega}(x)\right|\le C_{s'} e^{-\sqrt{\omega}|x|}.
\end{equation}

We can then ensure the existence of soliton solutions for \eqref{NLS}: given parameters
\[ \omega>0,  \quad \gamma\in\RR,  \quad x^0\in\RR^d,  \quad \text{and } v\in\RR^d,\]
the function
\[ R_{\omega,v,\gamma,x^0}(t,x) =  Q_{\omega}(x-x^0-vt)e^{i\big(\frac{1}{2}v\cdot x+\big(\omega-\frac{|v|^2}{4}\big)t+\gamma\big)}\]
satisfies \eqref{NLS}.

Concerning the construction of multi-solitons of \eqref{NLS}, let us recall

\begin{Th}[Merle \cite{merle_NLS_crit}; Martel and Merle \cite{martel_merle_NLS}, Côte, Martel and Merle \cite{cmm}; Côte and Friederich \cite{cf}]\label{th_multi_known_NLS}
Let $\displaystyle 1 < p < 1+\frac{4}{(d-2)_+}$. Let $N\ge 1$ and fix for all $j\in\{1,\dots,N\}$
\[ \omega_j>0,  \quad \gamma_j\in\RR,  \quad x_j^0\in\RR^d,  \quad \text{and } v_j\in\RR^d \text{ such that for all } j\ne j',\quad v_j\neq v_{j'}.\]
There exists $T_0\ge 0$ and a solution $u\in\mathscr{C}([T_0,+\infty),H^s(\RR))$ of \eqref{NLS} with $s = \max(1,\lfloor p-1 \rfloor)$, and a positive constant $\lambda_s$ such that
\begin{equation}\label{est_Hs_NLS}
 \left\|u(t)-\sum_{j=1}^NR_{\omega_j,v_j,\gamma_j,x_j^0}(t)\right\|_{H^s} \le \lambda_s e^{-\theta t} \quad\text{as } t\to +\infty.
\end{equation}
Moreover, if $p$ is an odd integer, then $u$ belongs to $\mathscr{C}([T_0,+\infty),H^s(\RR))$ and \eqref{est_Hs_NLS} holds for all $s \ge 0$. 
\end{Th}

\bigskip

We now state algebraic decay of the multi-solitons of \eqref{NLS} outside the solitons region.

\begin{Th}\label{th_decay_NLS}
Assume that $p$ is an odd integer such that $1<p<1+\frac{4}{(d-2)_+}$. Let $N\ge 1$ and fix for all $j\in\{1,\dots,N\}$
\[ \omega_j>0,  \quad \gamma_j\in\RR,  \quad x_j^0\in\RR^d,  \quad \text{and } v_j\in\RR^d \text{ such that for all } j\ne j',\quad v_j\neq v_{j'}, \]
and let $u$ be a multi-soliton associated with these parameters as in Theorem \ref{th_multi_known_NLS}.

Let $\beta>\max\{|v_j|,\;j=1,\dots,N\}$.  There exists $T_1>0$ such that for all $s\in\NN^d$, for all $n\in\NN$, there exists $C_{s,n}>0$ such that for all $t\ge T_1$,
\begin{equation}\label{dec_right_NLS}
\forall |x| > \beta t, \quad \left|\partial^s u(t,x)\right|\le \frac{C_{s,n}}{(|x|-\beta t)^n}.
\end{equation}
\end{Th}

\begin{Rq}
In the solitons region $\{ x : |x| \le \beta t \}$, one has exponential decay in a similar way as \eqref{dec_exp_mid}. Indeed, as a straighforward consequence of \eqref{est_Hs_NLS}, there exists $\kappa>0$ such that for all $s\in\NN^d$, there exists $C_s>0$ such that
\begin{equation}\label{exp_decay_solitons_region}
\forall |x| \le \beta t, \quad \left|\partial^s u(t,x)\right|\le \sum_{j=1}^Ne^{-\kappa|x-v_jt|}.
\end{equation}
\end{Rq}

\begin{Rq}\label{rq:generalization}
Theorems \ref{th_multi_known_NLS} and \ref{th_decay_NLS} also apply to nonlinear Schrödinger equations 
 \[ i\partial_t u+\Delta u+g(u)=0, \]
where $g:\CC\to \CC$ is a smooth nonlinearity, gauge invariant, that is of the form $g(z) = z f(|z|^2)$ and so that:
\begin{itemize}
\item there exists $p\in\left]1,1+\frac{4}{(d-2)_+}\right[$ such that for all $q\le p$, for all $r=0,\dots,q$, 
\[ \left|\frac{\partial^qg}{\partial_x^r\partial_y^{q-r}}(z)\right|=\mathcal{O}\left(|z|^{p-q}\right) \quad \text{as } |z|\to +\infty. \]
\item the frequencies $\omega_j$ are such that there is a $\mathcal C^1$ map $\mathcal W_j \to \mathscr{S} (\RR)$, $\omega \mapsto Q_\omega$ defined on some neighborhood $\mathcal W_j$ of $\omega_j$, such that
\[ \Delta Q_{\omega}+g(Q_{\omega})=\omega Q_{\omega}, \] 
We refer to Berestycki and Lions \cite{berestycki} for sufficient conditions on $g$ to ensure this condition.
\item the associated linearized operators around the $Q_{\omega_j}$
\[ \begin{array}{cccl}
\mathscr{L}_{\omega_j} : &H^1(\RR^d,\CC)& \to &H^1(\RR^d,\CC)\\
&v = v_1 + iv_2 & \mapsto &- \Delta v + \omega_j v - (f(Q_{\omega_j}^2) v+ 2 Q_{\omega_j}^2 f'(Q^2_{\omega_j}) v_1) 
\end{array} \]
satisfy suitable coercivity assumptions, as in \cite[Hypotheses (H3) and (H4)]{cf}. We refer to \cite[Proposition 1.7]{cf} and reference therein for a sufficient conditions in the case when $Q_{\omega_j}$ is a ground state to satisfy these coercivity conditions.
\end{itemize}
Under these assumptions, the $Q_{\omega_j}$ are exponentially decaying, along with their derivatives, and there exist $T_0$, $\theta>0$, and a multi-soliton $u\in\mathscr{C}([T_0,+\infty),H^\infty(\RR^d))$ of \eqref{NLS} such that for all $s\ge 0$, there exists $C_s>0$ such that
\begin{equation}
\forall t\ge T,\quad \left\|u(t)-\sum_{j=1}^N R_{\omega_j,v_j,\gamma_j,x_j^0}(t)\right\|_{H^s}\le C_se^{-\theta t}.
\end{equation} 
\end{Rq}

\subsection{Comments and strategy of the proof}

Theorems \ref{th_decay} and \ref{th_decay_NLS} show in particular that for each fixed time $t \ge T_1$, the multi-soliton $u(t)$ belongs to the Schwartz space $\mathscr{S}(\RR)$. To our knowledge, these are the first results of quantitative spatial decay in a non integrable setting. 

\bigskip

In \cite{friederich} and \cite[Section 3.1.2, (3.6) and (3.9)]{friederich_these}, the second author defines non dispersive solution of \eqref{gKdV} $u$ at $+\infty$ by the property that for some $\rho >0$,
\[ \int_{x \le \rho t} |u(t,x)|^2 \; dx \to 0 \quad \text{as} \quad t \to +\infty. \]
(such a notion was first developed by Martel and Merle \cite{asympstabcrit,asympstab0} in the vici\-nity of solitons). \cite{friederich} showed that non dispersion is a dynamical characterization of multi-solitons: more precisely, a solution of \eqref{gKdV} which is non dispersive and remains close to a sum of $N$ decoupled solitary waves for positive times is a multi-soliton in $+\infty$. For (KdV), the result is non perturbative: for any solutions (with  sufficiently smooth initial data), non dispersion is equivalent to being a multi-soliton (for (mKdV), breathers may also occur). The convergence \eqref{CVexp_Hs} shows on the other side that multi-solitons are non dispersive indeed. The decay obtained in Theorem \ref{th_decay} provides a quantitative version of this non dispersion. As far as we can tell, the classification of non dispersive solutions of \eqref{NLS} is not known.

\bigskip

Throughout the proofs, $u$ is as in the statement of Theorem \ref{th_decay} or Theorem \ref{th_decay_NLS}, depending on the equation we are studying, and it will be convenient to denote
\begin{equation} \label{def:Ri_z}
R_j:=R_{c_j,x_j}, \quad R:=\sum_{j=1}^NR_j,\quad \text{and} \quad z(t):=u(t)-R(t),
\end{equation}
in the case of the \eqref{gKdV} equation and to use analogous notations when considering \eqref{NLS}.

\bigskip

To prove Theorem \ref{th_decay}, we actually split space into three regions: the region to the left of the solitons, that is for $x \le \alpha t$ for some $\alpha <c_1$; the solitons region $\alpha t \le x \le \beta t$;  and the region to the right of the solitons $x \ge \beta t$.

Exponential decay of the multi-solitons of \eqref{gKdV} on the left of the solitons (for $x\le\alpha t$) follows from revisiting a monotonicity argument set up in \cite[section 2]{friederich} (strengthened from Laurent and Martel \cite{laurent}) and originally developed by Martel and Merle \cite{asympstabcrit}. We take advantage here of the convergence \eqref{def_multisol} and the decay of the solitons, instead of a non dispersion assumption, as it is done in the mentioned references.

In the solitons region $\alpha t \le x \le \beta t$, estimate \eqref{dec_exp_mid} is a direct consequence of the exponental convergence in \eqref{CVexp_Hs}.

The main novelty (and where most of our efforts are focused) concerns the region to the right of the solitons $x\ge\beta t$. The monotonicity argument, linked to the dynamic of the flow of \eqref{gKdV}, does not apply anymore: indeed, it would require some knowledge (non dispersion) at $t \to -\infty$ (or at least near the minimal existence time, as multi-solitons might blow up for the $L^2$-supercritical (gKdV)). From this perspective, the point of Theorem \ref{th_decay} is actually to obtain some information of the behavior of multi-solitons  for large decreasing times. Also notice that it would be sufficient to prove pointwise decay on the region $x\ge\beta t_0$ for one time $t_0$, and then this information would easily be propagated for $t \ge t_0$. This is in line with general statements linked with persistence of regularity and decay of solutions to \eqref{gKdV}, like Kato smoothing  in \cite{kato_smoothingeffect} or Isaza, Linares and Ponce \cite{isaza,isaza1}. Let us also mention \cite{Mun11}\footnote{We thank Y. Martel for pointing to us this reference, upon completion of this work.}, where some polynomial decay was obtained (see Lemma 7.4).

Our strategy in the region $x \ge \beta t$ is as follows. We consider families of integrals of the form
\[ I_{\varphi,s,x_0}(t):=\int_{x \ge \beta t}\left(\partial_x^s z\right)^2(t,x) \varphi(x-x_0-\beta t) \;dx \]
where $\varphi$ is a suitable weight function. We show that variations of $I_{\varphi,s,x_0}$ are essentially controlled by the $I_{\varphi',s',x_0}$ for $s' \in \{ 0, \dots, s+1 \}$, under the  induction hypothesis of an exponential decay in time. Then, by integrations in $t$ and then in $x_0$, together with \eqref{CVexp_Hs} (which provides the base case) and a triangular induction process, we show that we can bound $I_{\varphi,s,0}$ for $\varphi(y) = y^n$ for all $n \in \NN$.

\bigskip

For Theorem \ref{th_decay_NLS}, we develop a similar analysis using integrals of the form
\[  I_{\varphi,s,x_0}(t):= \sum_{\sigma \in \NN^d, |\sigma|=s} \int \left|\partial^\sigma z \right|^2(t,x) \varphi(|x|-x_0 - \beta t) \;dx. \]

When the nonlinearity is not smooth but merely $\mathscr{C}^s$, the multi-solitons still enjoy polynomial decay for the first derivatives, as it is clear from the proof (see Propositions \ref{procede_tri} and \ref{prop_P'_NLS}). The interested reader may compute the precise rates.

\bigskip

We expect that \eqref{gKdV} multi-solitons decay exponentially on the right as well, that is \eqref{dec_exp_mid} holds without the restriction $x \le \beta t$: this seems a natural conjecture as solitons are exponentially localized on both ends. Still, for the time being, estimate \eqref{dec_right} is meaningful; and similarly we conjecture that \eqref{NLS} multi-solitons decay exponentially fast in space.
In any case,  we believe that our strategy is robust and extends to prove rapid algebraic decay for the multi-solitons of other non linear dispersive models.
\bigskip

The article is organized as follows. We first study \eqref{gKdV} multi-solitons: in section 2, we consider the left region $x \le \alpha t$ and the solitons region $\alpha t \le x \le \beta t$; and in section 3, we focus on the right region $x \ge \beta t$. Then, in section 4, we turn to the case of the \eqref{NLS} multi-solitons and prove Theorem \ref{th_decay_NLS}. In the appendix, we provide
some bound on the $H^s$ norm of \eqref{gKdV} solitons and multi-solitons, and in particular, track the constant $\lambda_s$ in \eqref{CVexp_Hs}. 

\section{Decay of the (gKdV) multi-solitons on the left}

\subsection{Decay of the (gKdV) multi-solitons on the left of the first soliton}\label{subsection_left}

The goal of this paragraph is to prove

\begin{Prop}[Exponential decay in large time on the left of the first soliton] \label{prop:left}
Let $0<\alpha<c_1$  and $\kappa_\alpha\in\left(0,\frac{\sqrt{\alpha}}{2}\right)$.
There exists $T_1 \ge T_0$ such that for all $s\in\NN$, there exists $C_s>0$ such that for all $t\ge T_1$,
\begin{equation} \label{prop_dec_exp_left}
\forall x \le \alpha t, \quad \left|\partial^s_x u(t,x)\right|\le C_{s}e^{-\kappa_\alpha|x-\alpha t|}.
\end{equation}
\end{Prop}

\begin{Rq}
Using \eqref{CVexp_Hs}, one can easily see that the decay \eqref{prop_dec_exp_left} implies the one stated in \eqref{dec_exp_mid} in the region $x \le \alpha t$, with $\kappa = \kappa_\alpha$.
\end{Rq}

\begin{proof}
The proof follows the ideas of \cite{friederich} and \cite{laurent}. To reach the conclusion, we show the existence of $T_1\in\RR$ such that for each $s \in\NN$, there exists $K_s>0$ such that, with $\kappa:=2\kappa_\alpha$,
\begin{equation}\label{ineq_u_L2_k}
\forall t \ge T_1,\qquad \int_{x\le \alpha t}\left(\partial_x^s u(t,x)\right)^2e^{\kappa(\alpha t-x)}\;dx\le K_s.
\end{equation}

The first (and main) step is to obtain \eqref{ineq_u_L2_k} for $s=0$. For this, we claim a strong monotonicity property which is the purpose of Lemma \ref{lem_monot} and Lemma \ref{lem_limite} below. \\

\noindent Let us introduce, for some $\kappa>0$ to be determined later, the function $\varphi$ defined by
\[ \varphi(x)=\frac{1}{2}-\frac{1}{\pi}\arctan (e^{\kappa x}). \] 
It satisfies the following properties
\begin{align}\label{prop1_phi}
\exists \lambda_0>0,\:\forall x\in\RR,\qquad \lambda_0e^{-\kappa|x|} & <-\varphi '(x)<\frac{1}{\lambda_0}e^{-\kappa|x|}, \\
\label{prop2_phi}
\forall x\in\RR,\qquad |\varphi^{(3)}(x)| & \le -\kappa^2\varphi '(x). \\
\label{prop3_phi}
\exists \lambda_1>0,\:\forall x\ge 0,\qquad\lambda_1e^{-\kappa x} & \le \varphi(x).
\end{align}

\noindent Moreover, let us observe that 
\begin{equation}\label{egalite_integrale_tildeM}
\int_{x<\alpha t}u^2(t,x)e^{\kappa(\alpha t-x)}\;dx=\displaystyle\int_{x<0}u^2\left(t,x+\alpha t\right)e^{-\kappa x}\;dx,
\end{equation}
and that, for all $x_0<0$,
\begin{align} \nonumber
\int_{x_0\le x< 0}u^2\left(t,x+\alpha t\right)e^{-\kappa x}\;dx&\le e^{-\kappa x_0}\int_{x\ge x_0}u^2\left(t,x+\alpha t\right)e^{-\kappa (x-x_0)}\;dx\\
&\le \displaystyle\frac{1}{\lambda_1} e^{-\kappa x_0}\int_{\RR}u^2\left(t,x+\alpha t\right)\varphi(x-x_0)\;dx. \label{ineq1}
\end{align}

Since $\kappa^2<\alpha$, one can choose $\delta\in (0,\alpha-\kappa^2)$. We consider $T_1\in\RR$ to be determined later. Then, for fixed $t_0\ge T_1$ and $x_0\in\RR$, we define
\[ \begin{array}{crcl}
I_{(t_0,x_0)}:&[T_1,+\infty)&\rightarrow&\RR^+\\
&t&\mapsto&\displaystyle\int_\RR u^2(t,x+\alpha t)\varphi\big(x-x_0+\delta(t-t_0)\big)\;dx.
\end{array} \]

\noindent We have \begin{equation}
\forall t\ge T_1,\quad I_{(t_0,x_0)}(t)=\displaystyle\int_\RR u^2(t,x)\varphi\big(x-x_0+\delta(t-t_0)-\alpha t\big)\;dx,
\end{equation} 
so that by derivation with respect to $t$, we obtain 
\begin{equation}\label{derivee}
\begin{aligned}
\displaystyle\frac{dI_{(t_0,x_0)}}{dt}(t)=&\displaystyle-3\int_{\RR}u_x^2(t,x)\varphi '(\tilde{x})\;dx-(\alpha-\delta)\int_\RR u^2(t,x)\varphi '(\tilde{x})\;dx\\
&\displaystyle+\int_\RR u^2(t,x)\varphi^{(3)}(\tilde{x})\;dx+\frac{2p}{p+1}\int_\RR u^{p+1}(t,x)\varphi '(\tilde{x})\;dx,
\end{aligned}
\end{equation}
where $\tilde{x}:=x-x_0+\delta(t-t_0)-\alpha t$. We then claim

\begin{Lem}\label{lem_monot} 
There exists $C_0>0$ such that
\begin{equation}
\forall  x_0 \in \RR, \ \forall t_0,t \ge T_1, \quad \frac{dI_{(t_0,x_0)}}{dt}(t)\ge -C_0e^{-\kappa\left(-x_0+\delta(t-t_0)\right)}.
\end{equation}
\end{Lem}

\begin{proof}
Due to property (\ref{prop2_phi}) of $\varphi$, we have 
\begin{equation}\label{int0}
\left|\int_{\RR}u^2(t,x)\varphi^{(3)}(\tilde{x})\;dx\right|\le - \kappa^2\int_\RR u^2(t,x)\varphi '(\tilde{x})\;dx. 
\end{equation}

\noindent Furthermore we control the nonlinear part by considering 
\[ \displaystyle I_1(t):=\int_{|\tilde{x}|>-x_0+\delta(t-t_0)} u^{p+1}(t,x)\varphi '(\tilde{x})\;dx \] 
and 
\[ \displaystyle I_2(t):=\int_{|\tilde{x}|\le -x_0+\delta(t-t_0)} u^{p+1}(t,x)\varphi '(\tilde{x})\;dx. \]
On the one hand, we have due to (\ref{prop1_phi}) 
\begin{equation}\label{int_I1}
\displaystyle\big|I_1(t)\big| \le \frac{1}{\lambda_0}e^{-\kappa\big(-x_0+\delta(t-t_0)\big)}\int_\RR|u|^{p+1}(t,x)\;dx\le Ce^{-\kappa\big(-x_0+\delta(t-t_0)\big)},
\end{equation}
where we have used the Sobolev embedding $H^1(\RR)\hookrightarrow L^{p+1}(\RR)$ and the fact that $u$ belongs to $L^\infty([T_1,+\infty),H^1(\RR))$. Note that $C>0$ is independent of $x_0$, $t_0$, and $t$. \\
On the other, we observe that
\begin{equation}
\begin{aligned}
\displaystyle\big|I_2(t)\big|&\le \|u(t)\|_{L^\infty\big(x\le \alpha t\big)}^{p-1}\displaystyle\int_{x\le \alpha t}u^2(t,x)|\varphi '(\tilde{x})|\;dx\\
&\le \displaystyle\sqrt{2}^{p-1}\|u(t)\|_{L^2\big(x\le \alpha t\big)}^{\frac{p-1}{2}}\|u_x(t)\|_{L^2\big(x\le \alpha t\big)}^{\frac{p-1}{2}} \displaystyle\int_{\RR}u^2(t,x)|\varphi '(\tilde{x})|\;dx\\
&\le \displaystyle\sqrt{2}^{p-1}\|u(t)\|_{L^2\big(x\le \alpha t\big)}^{\frac{p-1}{2}}\sup_{t\ge T_1}\|u(t)\|_{H^1}^{\frac{p-1}{2}} \displaystyle\int_{\RR}u^2(t,x)|\varphi '(\tilde{x})|\;dx.
\end{aligned}
\end{equation}  

Since $u$ is a multi-soliton, we can choose $T_1\ge 0$ such that for all $t\ge T_1$, 
\begin{equation}\label{cond_T_1}
\sqrt{2}^{p-1}\|u(t)\|_{L^2\big(x\le \alpha t\big)}^{\frac{p-1}{2}}\sup_{t'\ge T_1}\|u(t')\|_{H^1}^{\frac{p-1}{2}}\le \frac{p+1}{2p}(\alpha-\delta-\kappa^2).
\end{equation}

Let us justify it briefly (here lies the main change with respect to previous proofs based on non dispersion \cite{friederich} or $L^2$-compactness \cite{laurent}): we have
\begin{align*}
\int_{x\le \alpha t}u^2(t,x)\;dx
&\le 2\int_{x\le \alpha t}\left(u-\sum_{j=1}^NR_j\right)^2(t,x)\;dx+2\int_{x\le \alpha t}\left(\sum_{j=1}^NR_j\right)^2(t,x)\;dx\\
&\le 2C_0^2e^{-2\theta t}+2N\sum_{j=1}^N\int_{x\le\alpha t}R_j^2(t,x)\;dx
\end{align*}

and for all $j=1,\dots,N$, since $\alpha<c_j$, we have for $t\ge 0$:
\begin{align*}
\MoveEqLeft \int_{x\le\alpha t}R_j^2(t,x)\;dx\le C\int_{x\le\alpha t}e^{-\sqrt{c_j}|x-c_jt-x_j|}e^{-\sqrt{c_j}|x-c_jt-x_j|}\;dx\\
&\le C\int_{x\le\alpha t}e^{-\sqrt{c_j}(c_j-\alpha)t}e^{-\sqrt{c_j}|x-c_jt-x_j|}\;dx\\
&\le Ce^{-\sqrt{c_j}(c_j-\alpha)t}\int_{\RR}e^{-\sqrt{c_j}|x-c_jt-x_j|}\;dx \le Ce^{-\sqrt{c_j}(c_j-\alpha)t}.
\end{align*}
where $C$ denotes a positive constant which can change from one line to the other and which only depends on $c_j$ (see expression \eqref{soliton1}). \\
Thus, we can pick up $C\ge 0$ such that for all $t\ge 0$,
\[ \int_{x\le \alpha t}u^2(t,x)\;dx\le C\left(e^{-2\theta t}+\sum_{j=1}^Ne^{-\sqrt{c_j}(c_j-\alpha)t}\right), \]
and then $T_1 \ge 0$ satisfying \eqref{cond_T_1}.

Taking into account (\ref{int_I1}), this eventually leads to the following estimate
\begin{align}\label{int}
\frac{2p}{p+1}\displaystyle\left|\int_\RR u^{p+1}(t,x)\varphi '(\tilde{x})\;dx\right| & \le -(\alpha-\delta-\kappa^2)\int_{\RR}u^2(t,x)\varphi '(\tilde{x})\;dx \nonumber \\
& \qquad +C_0e^{-\kappa\big(-x_0+\delta(t-t_0)\big)},
\end{align}
where $C_0:=\frac{2p}{p+1}C$ is independent of $x_0$, $t_0$, and $t$. Gathering \eqref{int0} and \eqref{int} in \eqref{derivee}, we finally deduce
\[ \frac{dI_{(t_0,x_0)}}{dt}(t)\ge -3\int_{\RR}u_x^2(t,x)\varphi '(\tilde{x})\;dx - C_0e^{-\kappa\big(-x_0+\delta(t-t_0)\big)}. \] 
 This establishes Lemma \ref{lem_monot}.
\end{proof}

As a consequence of the above lemma,
\begin{equation}\label{monot1}
\exists C_1>0,\:\forall x_0\in\RR,\:\forall t\ge t_0,\qquad I_{(t_0,x_0)}(t_0)\le I_{(t_0,x_0)}(t)+C_1e^{\kappa x_0},
\end{equation}
with $C_1$ independent of the parameters $x_0$ and $t_0$. Next, we claim the following:

\begin{Lem}\label{lem_limite}
For fixed $x_0\in\RR$ and $t_0\ge T_1$, $I_{(t_0,x_0)}(t)\rightarrow 0$ as $t\rightarrow +\infty$.
\end{Lem}

\begin{proof}
This lemma is shown by adapting the proof in \cite[paragraph 2.1, Step 2]{laurent} and in \cite{friederich}.
Let $\varepsilon$ be a positive real number.
As in the previous proof, because $u$ is a multi-soliton, we can find $T_1\ge T_0$ large such that for all $t\ge T_1$,
\[ \int_{x<\alpha t}u^2(t,x)\;dx\le \frac{\varepsilon}{2}.\] 
Since $0\le \varphi\le 1$, this enables us to see that 
\begin{align}
\displaystyle\int_{x<0}u^2\left(t,x+\alpha t\right)\varphi\big(x-x_0+\delta(t-t_0)\big)\;dx&\le \displaystyle\int_{x<\alpha t}u^2(t,x)\;dx \le \frac{\varepsilon}{2}.
\end{align}

\noindent Now, recall that $\varphi$ is decreasing so that 
\begin{align}
\MoveEqLeft[4]
\int_{x\ge 0}u^2(t,x+\alpha t)\varphi\big(x-x_0+\delta(t-t_0)\big)\;dx \nonumber\\
&\le \varphi\big(-x_0+\delta(t-t_0)\big)\|u(t)\|_{L^2}^2 \le \overline{C}\varphi\big(-x_0+\delta(t-t_0)\big),
\end{align} with $\overline{C}=\|u(t)\|_{L^2}^2$ for all $t\in J$. Moreover, since $\varphi(x)\to 0$ as $x\to +\infty$, there exists $T_2\in\RR$ such that for all $t\ge T_2$, 
\[ \overline{C}\varphi\big(-x_0+\delta(t-t_0)\big)\le \frac{\varepsilon}{2}. \]
Then, for all $t\ge \max\{T_1,T_2\}$,
\[ I_{(t_0,x_0)}(t)\le \frac{\varepsilon}{2}+\frac{\varepsilon}{2}=\varepsilon. \] 
Hence, we have finished proving Lemma \ref{lem_limite}.
\end{proof}

At this stage, we deduce from \eqref{monot1} and Lemma \ref{lem_limite} that
\begin{equation}\label{est_I_t_x}
\forall t_0\ge T_1,\:\forall x_0\in\RR,\quad I_{(t_0,x_0)}(t_0)\le C_1e^{\kappa x_0}.
\end{equation}
Thus, \eqref{ineq1} leads to
\[ \forall t\ge T_1,\quad \int_{x_0\le x<0}u^2\left(t,x+\alpha t\right)e^{-\kappa x}\;dx\le\frac{C_1}{\lambda_1}. \]
Letting $x_0\to -\infty$, we infer that 
\[ \forall t\ge T_1,\quad \int_{x<0}u^2\left(t,x+\alpha t\right)e^{-\kappa x}\;dx\le\frac{C_1}{\lambda_1}. \]
which proves \eqref{ineq_u_L2_k} with $s=0$.

\bigskip

Now, to conclude to \eqref{ineq_u_L2_k} for all $s \in\NN$,  one actually proves by induction on $s \in\NN$ the existence of $\tilde{K}_s\ge 0$ such that for all $t\ge T_1$, 
\begin{equation}\label{eq_rec_dec_u_k}
\int_{\RR}\left(\partial_x^s u\right)^2\left(t, x+\alpha t\right) e^{-\kappa x} \; dx + \int_{t}^{t+1}\int_{\RR}\left( \partial_x^s u \right)^2\left( \tau, x+\alpha \tau\right)e^{-\kappa x}\;dx\;d\tau\le\tilde{K}_s.
\end{equation}

For $s=0$, this is in fact a consequence of \eqref{ineq_u_L2_k} and of the following estimate: for all $t\ge t_0\ge T_1$,
\[ I_{(t_0,x_0)}(t_0)-I_{(t_0,x_0)}(t)\le \frac{C_1}{\lambda_1}e^{\kappa x_0}+3\int_{t_0}^{t}\int_\RR u_x^2(\tau,x+\alpha \tau)\varphi '(x-x_0+\delta(\tau-t_0))\;dx\;d\tau \]
 (which follows from the proof of Lemma \ref{lem_monot}).
Indeed, we notice that by \eqref{prop1_phi} and since $\varphi$ is decreasing, for $\tau\in [t_0,t]$,
\[ \lambda_0e^{-\kappa |x-x_0|}<-\varphi  '(x-x_0)\le -\varphi '(x-x_0+\delta(\tau-t_0)) \]
so that for $t=t_0+1$ in particular, we have
\begin{align*}
\MoveEqLeft \int_{t_0}^{t_0+1}\int_{x_0<x}u_x^2(\tau,x+\alpha \tau)e^{-\kappa x}\;dx\;d\tau \le Ce^{-\kappa x_0}\left(I_{(t_0,x_0)}(t)-I_{(t_0,x_0)}(t_0)\right)\\
&\le Ce^{-\kappa x_0}I_{(t_0,x_0)}(t_0+1) \le  Ce^{-\kappa x_0}I_{(t_0+1,x_0)}(t_0+1)\le C.
\end{align*}
where the last inequality results from \eqref{est_I_t_x}. Taking the limit when $x_0\to -\infty$, we obtain the desired inequality \eqref{eq_rec_dec_u_k}.

The rest of the induction argument closely follows \cite[paragraph 2.3 and paragraph 2.2 Step 2]{laurent}. Since it does not depend on the properties of the multi-soliton and for the sake of brevity, we will not detail the proof \eqref{eq_rec_dec_u_k} for higher values of $s$.
\end{proof}

\subsection{Decay of the (gKdV) multi-solitons in the solitons region}

\begin{Prop}[Exponential decay in the solitons region] \label{prop:soliton}
Let  $0<\alpha<c_1$ and $\beta>c_N$, and define
\begin{equation} \label{def:kappa_ab}
\kappa_{\alpha,\beta}:=\min\left\{\sqrt{c_1},\frac{\theta}{c_1-\alpha},\frac{\theta}{\beta-c_N},\min_{j=1,\dots,N-1}\left\{\frac{\theta}{c_{j+1}-c_j}\right\}\right\} >0.
\end{equation}
Then for all $s\in\NN$, there exists $C_s>0$ such that for all $t\ge T_0$,
\begin{equation}\label{prop_dec_exp_sol}
\forall x \in [\alpha t, \beta t], \quad \left|\partial^s_x u(t,x)\right|\le C_{s}\sum_{j=1}^Ne^{-\kappa_{\alpha,\beta}|x-c_jt|}.
\end{equation}
\end{Prop}

\begin{proof}
Recall the notation $z$ given in \eqref{def:Ri_z}. For all $s\in\NN$, for all $t\ge T_0$, we have by \eqref{CVexp_Hs} and the Sobolev embedding $H^1(\RR)\hookrightarrow L^\infty(\RR)$,
\[ \|\partial_x^sz(t)\|_{L^\infty}\le C\|\partial_x^sz(t)\|_{H^1}\le C\|z(t)\|_{H^{s+1}}\le \lambda_{s+1}e^{-\theta t}. \]
Fix $t\ge T_0$. For all $j=1,\dots,N-1$ and $c_j t\le x\le c_{j+1}t$, 
\[ e^{-\theta t}\le e^{-\kappa_{\alpha,\beta}(x-c_j t)}\quad\text{if and only if}\quad x\le \left(c_j+\frac{\theta}{\kappa_{\alpha,\beta}}\right)t, \]
which is indeed satisfied since $c_{j+1}\le c_j+\frac{\theta}{\kappa_{\alpha,\beta}}$ by the choice of $\kappa_{\alpha,\beta}$ \eqref{def:kappa_ab}.

Similarly, for all $\alpha t\le x\le c_{1}t$, we have $e^{-\theta t}\le e^{-\kappa_{\alpha,\beta}(c_1 t-x)}$ because $c_{1}\le \alpha+\frac{\theta}{\kappa_{\alpha,\beta}}$. 
And for all $c_N t\le x\le \beta t$, we have $e^{-\theta t}\le e^{-\kappa_{\alpha,\beta}(x-c_Nt)}$ because $\beta\le c_N+\frac{\theta}{\kappa_{\alpha,\beta}}$. 

Thus, we obtain that for all $t\ge T_0$ and for all $\alpha t\le x\le \beta t$, 
\[ \left|\partial_x^sz(t,x)\right|\le \|\partial_x^sz(t)\|_{L^\infty}\le \lambda_{s+1}e^{-\theta t}\le \lambda_{s+1}\sum_{j=1}^Ne^{-\kappa_{\alpha,\beta}|x-c_j t|}. \]
Moreover, for all $j=1,\dots,N$, $\left|\partial_x^sR_j(t,x)\right|\le C_{j,s}e^{-\sqrt{c_j}|x-c_jt|}$ for some $C_{j,s}>0$ depending on $j$ and $s$. Hence, we conclude to \eqref{prop_dec_exp_sol} by the triangular inequality and the fact that $\kappa_{\alpha,\beta}\le\sqrt{c_1}$.
\end{proof}

\section{Decay of the (gKdV) multi-solitons on the right of the last soliton}\label{sec_decay_gkdv_right}

In this subsection, we analyze the behavior of the multi-solitons on the right of the solitons region. We will focus on the proof of the following

\begin{Prop}[Polynomial decay in large time on the right of the last soliton]\label{prop_decroissance}
Let $\beta >c_N$. For all $s\in\NN$ and for all $n\in\NN$, there exists $C_{s,n}>0$ such that for all $t\ge T_0$, for all $x>\beta t$,
\begin{equation}
\left(\partial^s_x u(t,x)\right)^2\le \frac{C_{s,n}}{(x-\beta t)^n}.
\end{equation}
\end{Prop}

Notice that this statement will appear as a corollary of a more general result, which has its own interest (see Proposition \ref{prop_stability} below) and which relies on a triangular induction process.

By Theorem \ref{th_multi_known}, there exists $\theta>0$ such that for all $s\in\NN$, there exists $\lambda_s>0$ such that for all $t\ge T_0$,
\begin{equation}\label{est_Hs}
\|z(t)\|_{H^s}\le \lambda_se^{-\theta t}.
\end{equation}

One can estimate the growth of $\lambda_s$ with respect to $s$, and in fact, we will keep for the sequel
\begin{equation}\label{est_lambda_s}
\lambda_s\le C.2^{\mu_0^s},
\end{equation}
with $\mu_0>\max\left\{\sqrt{p},\frac{p+1}{2}\right\}$ defined in Claim \ref{claim_interaction}. 
The proof of estimate \eqref{est_lambda_s} is postponed to  paragraph \ref{app:a3} in the Appendix.

\subsection{The key ingredient: stability by integration of a well-chosen set of weight functions}

Fix
\[ \eta\in (0,c_1). \] The choice of $\eta$ is made in order to obtain the interaction estimate of Claim \ref{claim_interaction} below, which roughly expresses that the growth of $x\mapsto e^{\sqrt{\eta}x}$ is weaker than the decay of the solitons. 

\begin{Claim}\label{claim_interaction}
For all $s \in\NN$, there exists $C>0$ and $\mu>\sqrt{p}$ such that for all $j=1,\dots,N$ and for all $t\ge T_0$,
\begin{equation}\label{ineg_interaction}
\int_{\RR}\left|\partial_{x}^s R_j(t,x)\right|e^{\sqrt{\eta}(x-\beta t)}\;dx \le C 2^{\mu^s}.
\end{equation}
\end{Claim}

We will consider weight functions $\varphi\in\mathscr{C}^3(\RR,\RR)$ which satisfy the following assumptions:

\begin{align}
&(i) \quad\displaystyle\lim_{x\to -\infty}\varphi(x)=0 \tag{$\mathbf{A}$}\\
&(ii) \quad\exists \kappa_1>0,\:\forall x\in\RR,\quad 0\le \varphi '(x)\le \kappa_1e^{\sqrt{\eta} x}\notag\\
&(iii)\quad\exists \kappa_2>0,\:\forall x\in\RR,\quad |\varphi ^{(3)}(x)|\le \kappa_2 \varphi '(x)\notag
\end{align}

and for given $\overline{s}\in\NN$,

\begin{align}
&\exists C({\overline{s}},\varphi)>0,\:\forall s \in\{0,\dots,\overline{s}\},\forall t\ge T_0,\quad \int_{\RR}\left(\partial_x^s z\right)^2(t,x)\varphi(x-\beta t)\;dx\le C(\overline{s},\varphi)e^{-\theta t}.\tag{$\mathbf{B}(\overline{s})$}
\end{align}

Let us define the following property which depends on $\overline{s}$ and $\varphi$:
\[ P(\overline{s},\varphi):\quad \varphi \text{ satisfies } (\mathbf{A}) \text{ and } (\mathbf{B}(\overline{s})). \]

Note that under assumptions ($\mathbf{A}$) ($i$) and ($ii$), $\varphi$ is integrable in the neighborhood of $-\infty$ (see also \eqref{ineg_varphi} below); we can thus define on $\RR$ the antiderivative $\varphi_{[1]}$ of $\varphi$:
\[ \varphi_{[1]}:x\mapsto \int_{-\infty}^x\varphi(r)\;dr. \]

Next, we state the key ingredient, reflecting the triangular way of obtaining Proposition \ref{prop_decroissance},

\begin{Prop}\label{procede_tri}
For $\overline{s}\in\NN$, $P(\overline{s}+1,\varphi)\Rightarrow P(\overline{s},\varphi_{[1]})$ and moreover, given $\mu_0>\max\left\{\sqrt{p},\frac{p+1}{2}\right\}$, there exists a constant $c(\eta,\kappa_1,\kappa_2)>0$ (which depends only on $\eta,\kappa_1$ and $\kappa_2$) such that
\begin{equation}\label{ineg_procede_tri}
C(\overline{s},\varphi_{[1]})\le c(\eta,\kappa_1,\kappa_2)2^{\mu_0^{\overline{s}}}C(\overline{s}+1,\varphi).
\end{equation}
\end{Prop}

As a corollary of the previous proposition, defining the set
\[ \mathcal{E}:=\left\{\varphi\in\mathscr{C}^3(\RR,\RR)\:|\: \varphi \text{ satisfies } (\mathbf{A}) \text{ and } (\mathbf{B}(s)) \text{ for all } s\in\NN\right\}, \]
we immediately obtain step by step, in a triangular way, the following

\begin{Prop}[Stability by integration]\label{prop_stability}
If $\varphi\in\mathcal{E}$, then $\varphi_{[1]}\in\mathcal{E}$.
\end{Prop}

\begin{Rq}\label{rq_exp}
Proposition \ref{prop_stability} is enough to prove that the multi-soliton and its derivatives have polynomial decay (see subsection \ref{subsec_rapid_dec}).
Note that if one could improve \eqref{est_lambda_s} and \eqref{ineg_procede_tri} by proving the existence of $C>0$ and $c(\eta,\kappa_1,\kappa_2)>0$ such that
\[ \forall s\in\NN,\quad\lambda_s\le C^s \]
and
\[ \forall s\in\NN,\:\forall\varphi\in\mathcal{E},\quad C(s,\varphi_{[1]})\le c(\eta,\kappa_1,\kappa_2)C(s+1,\varphi), \]
we would deduce that the multi-soliton and all its derivatives decay exponentially on the domain $x>\beta t$ (see paragraph \ref{subs_appendix_3} in the Appendix). 
\end{Rq}

\subsection{Proof of Proposition \ref{procede_tri}}

\begin{proof}[Proof of Claim \ref{claim_interaction}]
On the one hand, we have
\begin{align*}
\int_{x\le\beta t}e^{-\sqrt{c_j}|x-c_jt|}e^{\sqrt{\eta}(x-\beta t)}\;dx&\le \int_{x\le\beta t}e^{-\sqrt{c_j}|x-c_jt|}\;dx \le \int_{\RR}e^{-\sqrt{c_j}|x|}\;dx\le \frac{2}{\sqrt{c_j}}.
\end{align*}
On the other hand, since $\beta>c_N$ and $\beta\sqrt{\eta}<c_j\sqrt{c_j}$,
\begin{align*}
\int_{x>\beta t}e^{-\sqrt{c_j}|x-c_jt|}e^{\sqrt{\eta}(x-\beta t)}\;dx&\le e^{(c_j\sqrt{c_j}-\beta\sqrt{\eta})t}\int_{x> \beta t}e^{(\sqrt{\eta}-\sqrt{c_j})x}\;dx\\
&\le \int_{x> \beta t}e^{(\sqrt{\eta}-\sqrt{c_j})x}\;dx\le \frac{e^{(\sqrt{\eta}-\sqrt{c_j})\beta t}}{\sqrt{c_j}-\sqrt{\eta}}.
\end{align*}
Hence, noticing that we have also $\left|\partial_{x}^kR_j(t,x)\right|\le C_{j,k} e^{-\sqrt{c_j}|x-c_jt|}$, where $C_{j,k}$ is a constant depending on the parameters of the soliton $R_j$ and on $k$ only, Claim \ref{claim_interaction} holds.
\end{proof}

\begin{proof}[Proof of Proposition \ref{procede_tri}]
First of all, let us check that the properties gathered in ($\mathbf{A}$) are stable by integration, that is, if we assume that $\varphi$ satisfies ($\mathbf{A}$), then so does $\varphi_{[1]}$. \\
Assumption ($ii$) shows that $\varphi '$ is integrable in the neighborhood of $-\infty$ and 
\[ \forall x\in\RR,\quad 0\le \int_{-\infty}^x\varphi '(r)\;dr\le \kappa_1\int_{-\infty}^xe^{\sqrt{\eta}r}\;dr. \]
By (i), we thus obtain 
\begin{equation}\label{ineg_varphi}
\forall x\in\RR,\quad 0\le \varphi(x)\le \frac{\kappa_1}{\sqrt{\eta}}e^{\sqrt{\eta}x}.
\end{equation}
Then ($iii$) implies that $\varphi^{(3)}$ is integrable in the neighborhood of $-\infty$ and so $\varphi ''$ admits a limit in $-\infty$, which is necessarily 0 (since $\varphi '(x)\to 0$ as $x\to -\infty$). Finally, by integration, one obtains
\[ \forall x\in\RR,\quad |\varphi ''(x)|\le \kappa_2\varphi(x). \]
Hence $\varphi_{[1]}$ indeed satisfies ($\mathbf{A}$). \\

Now take $\overline{s}\in\NN$ and let us show that $\varphi_{[1]}$ verifies $\mathbf{B}(\overline{s})$ if one assumes that $\varphi$ satisfies $\mathbf{B}(\overline{s}+1)$.\\

We define for all $s\in\NN$, for all $x_0\ge 0$ and for all $t\ge T_0$:
\[ J_{s,x_0}(t):=\int_{\RR}\left(\partial_x^sz\right)^2(t,x)\varphi(x-x_0-\beta t)\;dx. \]

For ease of reading, we will denote $\tilde{x}=\tilde{x}(t):=x-x_0-\beta t$ (for $x_0>0$ and $t\ge T_0$). 

We first show the following induction formula which makes the link between the functions $J_{s,x_0}$, $s\in\NN$.

\begin{Lem}\label{lem_der_I}
For all $s\in\NN$, there exists $C_s\ge 0$ (independent of $x_0$) such that for all $t \ge T_0$:
\begin{equation}\label{est_der_I_1}
\left|\frac{d}{dt}J_{s,x_0}(t)\right|\le C_s\int_{\RR}\sum_{k=0}^{s+1}\left(\partial_x^{k}z\right)^2(t,x)\varphi'(\tilde{x})\;dx +C_se^{-\theta t}\sum_{k=0}^{s-1}J_{k,x_0}(t)+C_se^{-\sqrt{\eta}x_0}e^{-\theta t}.
\end{equation}
In addition, for all $\mu_1>\max\left\{\sqrt{p},\frac{p+1}{2}\right\}$, there exists $\gamma_1>0$ independent of $s$ such that for all $s$,
\begin{align} \label{est:Cs}
C_s\le \gamma_1 2^{\mu_1^s}.
\end{align}
\end{Lem}

\begin{proof}
Let us compute 
\begin{align*}
\MoveEqLeft \frac{d}{dt}J_{s,x_0}(t)= -3\int_{\RR}\left(\partial_x^{s+1}z\right)^2\varphi'(\tilde{x})\;dx +\int_{\RR}\left(\partial_x^sz\right)^2\varphi^{(3)}(\tilde{x})\;dx \\
& -\beta\int_{\RR}\left(\partial_x^s z\right)^2\varphi'(\tilde{x})\;dx +2\int_{\RR}\partial_x^s\left(\left(z+R\right)^p-\sum_{j=1}^NR_j^p\right)\left(\partial_x^sz\varphi\right)_x\;dx.
\end{align*}

By ($\mathbf{A}$) ($iii$), we have 
\begin{equation}
\left|\int_{\RR}\left(\partial_x^sz\right)^2\varphi^{(3)}(\tilde{x})\;dx-\beta\int_{\RR}\left(\partial_x^s z\right)^2\varphi'(\tilde{x})\;dx\right|\le (\kappa_2+\beta) \int_{\RR}\left(\partial_x^s z\right)^2\varphi'(\tilde{x})\;dx.\end{equation}
We now control the nonlinear term $\int_{\RR}\partial_x^s\left(z^p\right)\left(\partial_x^sz\varphi\right)_x\;dx$ which does not contain any soliton. If $s=0$, we observe that
\[
\int_{\RR}z^p\left(z\varphi\right)'\;dx=\frac{p}{p+1}\int_{\RR}z^{p+1}\varphi'(\tilde{x})\;dx;
\]
thus \begin{equation}
\left|\int_{\RR}z^p\left(z\varphi\right)'\;dx\right|\le \frac{p}{p+1}\|z(t)\|_{L^\infty}^{p-1}\int_\RR z^2\varphi'(\tilde{x})\;dx\le C\int_\RR z^2\varphi'(\tilde{x})\;dx.
\end{equation}
If $s\ge 1$, we can write 
\[\int_{\RR}\partial_x^s\left(z^p\right)\left(\partial_x^sz\varphi\right)_x\;dx=\int_{\RR}\partial_x^s\left(z^p\right)\partial_x^sz\varphi'(\tilde{x})\;dx -\int_{\RR}\partial_x^{s-1}\left(z^p\right)\left(\partial_x^{s+2}z\varphi+\partial_x^{s+1}z\varphi'\right)\;dx.
\]
We have
\[ \partial^k_x\left(z^p\right)=\sum_{i_1+\dots+i_p=k}\dbinom{k}{i_1,\dots,i_p}\partial_x^{(i_1)}z\dots\partial_x^{(i_p)}z \]
so that 
\begin{align}\label{est_lem_1}
\left|\int_\RR\partial_x^{s-1}\left(z^p\right)\partial_x^{s+2}z\varphi\;dx\right| & \le \|\partial_x^{s+2}z(t)\|_{L^\infty}\sum_{i_1+\dots+i_p=s-1}\int_{\RR}\left|\partial_x^{(i_1)}z\right|\dots\left|\partial_x^{(i_p)}z\right|\varphi(\tilde{x})\;dx \nonumber \\
&\le C\|z(t)\|_{H^{s+3}}\frac{1}{p}\sum_{i_1+\dots+i_p=s-1}\sum_{k=1}^{p}\int_\RR\left|\partial_x^{(i_k)}z\right|^p\varphi(\tilde{x})\;dx \nonumber\\
&\le C\|z(t)\|_{H^{s+3}}\|z(t)\|_{H^{s}}^{p-2}\sum_{i_1+\dots+i_p=s-1}\sum_{k=1}^{p}\int_\RR\left(\partial_x^{(i_k)}z\right)^2\varphi(\tilde{x})\;dx \nonumber \\
&\le Cp^s\lambda_{s+3}\lambda_s^{p-2}e^{-(p-1)\theta t}\sum_{k=0}^{s-1}\int_\RR\left(\partial_x^kz\right)^2\varphi(\tilde{x})\;dx,
\end{align}
where we used \eqref{est_Hs}. Similarly we obtain
\begin{equation}\label{est_lem_2}
\left|\int_\RR\partial_x^{s-1}\left(z^p\right)\partial_x^{s+1}z\varphi'\;dx\right|\le Cp^s\lambda_{s+2}\lambda_s^{p-2}e^{-(p-1)\theta t}\int_\RR\sum_{k=0}^{s-1}\left(\partial_x^{k}z\right)^2\varphi'(\tilde{x})\;dx
\end{equation}
and 
\begin{equation}\label{est_lem_3}
\left|\int_\RR\partial_x^{s}\left(z^p\right)\partial_x^{s}z\varphi'\;dx\right|\le Cp^s\lambda_{s+2}\lambda_{s+1}^{p-2}e^{-(p-1)\theta t}\int_\RR\sum_{k=0}^{s}\left(\partial_x^{k}z\right)^2\varphi'(\tilde{x})\;dx.
\end{equation}

Hence we can take
\[ C_s\le Cp^s\lambda_{s+3}^{p-1}\le 2^{\mu_1^s} \]
 for all $\mu_1>\mu_0>\max\left\{\sqrt{p},\frac{p+1}{2}\right\}$ and $s$ large enough.

Moreover,
\[ \int_{\RR}\partial_x^s\left(\left(z+R\right)^p-\sum_{j=1}^NR_j^p-z^p\right)\left(\partial_x^sz\varphi\right)_x\;dx=I_1+I_2, \]
with 
\[ I_1=\int_\RR\partial_x^s\left(R^p-\sum_{j=1}^NR_j^p\right)\left(\partial_x^{s+1}z\varphi+\partial_x^sz\varphi '\right)\;dx \]
and
\[ I_2=\sum_{k=1}^{p-1}\dbinom{p}{k}\sum_{i_1+\dots+i_p=s}\dbinom{s}{i_1,\dots,i_p}\int_\RR\partial_x^{i_1}z\dots\partial_x^{i_k}z\partial_x^{i_{k+1}}R\dots\partial_x^{i_p}R\left(\partial_x^{s+1}z\varphi+\partial_x^sz\varphi '\right)\;dx. \]
We have
\begin{align*}
|I_1|&\le \left(\int_\RR\left(\partial_x^{s+1}z\right)^2\;dx\right)^{\frac{1}{2}}\left(\int_\RR\left(\partial_x^{s}\left(R^p-\sum_{j=1}^NR_j^p\right)\varphi\right)^2\;dx\right)^{\frac{1}{2}}\\
& \qquad +\left(\int_\RR\left(\partial_x^{s}z\right)^2\;dx\right)^{\frac{1}{2}}\left(\int_\RR\left(\partial_x^{s}\left(R^p-\sum_{j=1}^NR_j^p\right)\varphi '\right)^2\;dx\right)^{\frac{1}{2}}\\
&\le C\|z\|_{H^{s+1}}p^{s+1}2^{\mu^s}e^{-2\theta t}e^{-\sqrt{\eta}x_0}.
\end{align*}
Moreover
\begin{equation}\label{est_lem_4}
\begin{aligned}
|I_2|&\le \sum_{k=1}^{p-1}\dbinom{p}{k}\sum_{i_1+\dots+i_p=s}\dbinom{s}{i_1,\dots,i_p}\|z\|_{H^{s+2}}\|z\|_{H^{s+1}}^{k}\|R\|_{H^{s+1}}^{p-k-1}\int_\RR\left(\left|\partial_x^{i_p}R\varphi\right|+\left|\partial_x^{i_p}R\varphi '\right|\right)\;dx\\
&\le C2^pp^s\|z\|_{H^{s+2}}^{k+1}\|R\|_{H^{s+1}}^{p-k-1}
p^{s+1}2^{\mu^s}e^{-2\theta t}e^{-\sqrt{\eta}x_0},
\end{aligned}
\end{equation}
where the second inequality is a consequence of Claim \ref{claim_interaction}. Indeed, Claim \ref{claim_interaction} rewrites as follows: for all $x_0\ge 0$,
\[ \int_{\RR}\left|\partial_{x}^kR_j(t,x)\right|e^{\sqrt{\eta}(x-x_0-\beta t)}\;dx\le Ce^{-\sqrt{\eta}x_0}. \]
Thus, by property ($\mathbf{A}$) $(ii)$ satisfied by $\varphi$, we infer
\[ \int_{\RR}\left|\partial_{x}^kR_j(t,x)\right|\left(\varphi(\tilde{x})+\varphi '(\tilde{x})\right)\;dx\le Ce^{-\sqrt{\eta}x_0}. \]

Now, we obtain Lemma \ref{lem_der_I} by gathering the above estimates. We can find a constant $\gamma_1$ independent of $s$ and depending only on $\eta$, $\kappa_1$ and $\kappa_2$ such that for $s$ sufficiently large,
\[ C_s\le \gamma_12^{\mu_1^s}. \]
Even if it means taking $\gamma_1$ greater, we can assume that the above estimate holds for all $s$.
\end{proof}

\begin{Rq}
Let us observe that we could obtain sharper estimates than \eqref{est_lem_1}, \eqref{est_lem_2}, \eqref{est_lem_3}, and \eqref{est_lem_4} due to integrations by parts. But this would have only little impact on the growth rate in $s$ at this stage, and in the end, it would not improve \eqref{ineg_procede_tri}. 
\end{Rq}

Then, we obtain the following control of $J_{s,x_0}(t)$:

\begin{Lem}\label{lem_est_I}
For all $s\in \NN$, there exists a constant $K_s\ge 1$ such that for all $t\ge T_0$:
\begin{equation}\label{est_I_2}
J_{s,x_0}(t)\le K_s\int_{t}^{+\infty}\int_{\RR}\left(\sum_{k=0}^{s+1}\left(\partial_x^{k}z\right)^2(t',x)\varphi'(\tilde{x}(t'))\right)\;dx\;dt'+K_se^{-\sqrt{\eta}x_0}e^{-\theta t}.
\end{equation}
In addition, given $\mu_2>\max\left\{\sqrt{p},\frac{p+1}{2}\right\}$, there exists $\gamma_2>0$ such that for all $s$,
\[ K_s\le \gamma_22^{\mu_2^s}. \]
\end{Lem}

\begin{proof}
It follows from \eqref{est_der_I_1} and an induction argument. 
Notice that for all $s\in \NN$, $J_{s,x_0}(t) \to 0$ as $t \to +\infty$. Thus, for $s=0$, \eqref{est_I_2} follows by integration of \eqref{est_der_I_1} between $t$ and $+\infty$.
Now assume that \eqref{est_I_2} is proved for $0,\dots,s-1$ for some particular $s\ge 1$. Then, by integration of \eqref{est_der_I_1} between $t$ and $+\infty$ (for $t\ge T_0$), it results:
\begin{align*}
J_{s,x_0}(t)\le& \; C_s\int_{t}^{+\infty}\int_{\RR}\sum_{k=0}^{s+1}\left(\partial_x^{k}z\right)^2(t',x)\varphi'(\tilde{x})\;dx\;dt'+C_se^{-\sqrt{\eta}x_0}\int_t^{+\infty}e^{-\theta t'}dt'\\
&\qquad +C_s\sum_{s'=0}^{s-1}K_{s'}\int_t^{+\infty}e^{-\theta t'}\int_{t'}^{+\infty}\int_{\RR}\sum_{k=0}^{s'+1}\left(\partial_x^{k}z\right)^2(t'',x)\varphi'(\tilde{x})\;dx\;dt''\;dt'\\
&\qquad +C_s\sum_{s'=0}^{s-1}K_{s'}e^{-\sqrt{\eta}x_0}\int_t^{+\infty}e^{-\theta t'}\;dt' \\
& \le C_s\int_{t}^{+\infty}\int_{\RR}\left(\sum_{k=0}^{s+1}\left(\partial_x^{k}z\right)^2(t',x)\varphi'(\tilde{x})\right)\;dx\;dt'\\
&\qquad +\sum_{s'=0}^{s-1}C_sK_{s'}\left(\int_t^{+\infty}\int_\RR\sum_{k=0}^{s'+1}\left(\partial_x^{k}z\right)^2(t'',x)\varphi'(\tilde{x})\;dx\;dt''\right)\int_t^{+\infty}e^{-\theta t'}dt'\\
&\qquad + \frac{C_s}{\theta} \max\{K_{s'},s'=0,\dots,s-1\} e^{-\sqrt{\eta}x_0}e^{-\theta t}.
\end{align*}
Hence there exists $K_{s}\ge 1$ such that
\begin{align}
\label{est:Ks}
K_s\le C_s+\sum_{s'=0}^{s-1}C_sK_{s'} \le 2C_s  \sum_{s'=0}^{s-1}K_{s'}
\end{align}
and for which
\[ J_{s,x_0}(t)\le K_s\int_{t}^{+\infty}\int_{\RR}\left(\sum_{k=0}^{s+1}\left(\partial_x^{k}z\right)^2(t',x)\varphi'(\tilde{x})\right)\;dx\;dt'+K_se^{-\sqrt{\eta}x_0}e^{-\theta t}. \]
From the inequality \eqref{est:Ks} and an induction argument, we can bound
\begin{align*}
K_s& \le 2C_s \left( \sum_{s'=0}^{s-2} K_{s'} + K_{s-1} \right) \le  2C_s (1 + 2C_{s-1})  \sum_{s'=0}^{s-2} K_{s'} \\
&\le 2C_s(1+2C_{s-1})\dots (1+2C_1)K_0 \\
& \le 2C_s\times 4C_{s-1}\times\dots\times 4C_1C_0 \le 2^{2s-1}\prod_{i=0}^s C_i \le 2^{\mu_2^s}, 
\end{align*}
for all $\mu_2>\mu_1$ and $s$ sufficiently large (see \eqref{est:Cs} in Lemma \ref{lem_der_I}).
\end{proof}

Let us now conclude the proof of $P(\overline{s},\varphi_{[1]})$. We integrate estimate \eqref{est_I_2} provided by Lemma \ref{lem_est_I} on $[0,+\infty)$ with respect to $x_0$.
We obtain by Fubini theorem: for $t\ge T_0$,
\begin{align*}
\MoveEqLeft
\int_\RR\left(\partial_x^{\overline{s}}z\right)^2(t,x)\int_{0}^{+\infty}\varphi(x-x_0-\beta t)\;dx_0\;dx\\
& \le K_{\overline{s}}\int_t^{+\infty}\int_\RR\sum_{k=0}^{\overline{s}+1}\left(\partial_x^kz\right)^2(t',x)\int_{0}^{+\infty}\varphi '(x-x_0-\beta t')\;dx_0\;dx\;dt' +\frac{K_{\overline{s}}}{\sqrt{\eta}}e^{-\theta t}
\end{align*} and then by an affine change of variable
\begin{align*}
\int_\RR\left(\partial_x^{\overline{s}}z\right)^2(t,x)\varphi_{[1]}(x-\beta t)\;dx \le K_{\overline{s}}\int_t^{+\infty}\int_\RR\sum_{k=0}^{\overline{s}+1}\left(\partial_x^kz\right)^2(t',x)\varphi(x-\beta t')\;dx\;dt' +\frac{K_{\overline{s}}}{\sqrt{\eta}}e^{-\theta t}.
\end{align*}
Considering that $\varphi$ satisfies $(\mathbf{B}(\overline{s}+1))$, this finally shows that
\[ \int_\RR\left(\partial_x^{\overline{s}}z\right)^2(t,x)\varphi_{[1]}(x-\beta t)\;dx\le \frac{K_{\overline{s}}}{\theta}\sum_{k=0}^{\overline{s}+1}C(k,\varphi)e^{-\theta t}+\frac{K_{\overline{s}}}{\sqrt{\eta}}e^{-\theta t} \]
Hence, $\varphi_{[1]}$ satisfies $(\mathbf{B}(\overline{s}))$ and one can take
\[ C(\overline{s},\varphi_{[1]})\le C (\overline{s}+2)K_{\overline{s}}C(\overline{s}+1,\varphi). \]

Thus we obtain \eqref{ineg_procede_tri}, which finishes proving Proposition \ref{procede_tri}. 
\end{proof}

\subsection{Rapid decrease on the right: proof of Proposition \ref{prop_decroissance}}\label{subsec_rapid_dec}

\begin{proof}
Now, we show polynomial decay of $z$ and its derivatives. This consists in an application of Proposition \ref{prop_stability} and is the object of Claim \ref{claim_poly} and Claim \ref{claim_est_z} below.

Set $\eta\in\left(0,c_1\right)$ and introduce the function $\varphi:\RR\to\RR$ defined by 
\[ \varphi(x):=\frac{2}{\pi}\arctan\left(e^{\sqrt{\eta}x}\right). \]
The precise form of $\varphi$ is not that important, but this expression is convenient. Observe that $\varphi\in\mathcal{E}$, in view of \eqref{est_Hs} and due to $\varphi$ being bounded. We define a sequence $\left(\varphi_{[n]}\right)_{n\in\NN}$ of functions $\RR\to\RR$ as follows: $\varphi_{[0]}:=\varphi$ and for all $n\in\NN^*$, for all $x\in\RR$,
\[ \varphi_{[n]}(x):=\int_{-\infty}^x\varphi_{[n-1]}(y)\;dy. \]
By Proposition \ref{prop_stability}, we have 
\[ \forall n\in\NN,\qquad \varphi_{[n]}\in\mathcal{E}. \]
The following claim motivates the introduction of this sequence $\left(\varphi_{[n]}\right)_{n}$.

\begin{Claim}[Polynomial growth of $\varphi_{[n]}$]\label{claim_poly}
We have for all $n\in\NN$
\begin{equation}\label{eq_poly_n_1}
\forall x\le 0,\qquad 0\le \varphi_{[n]}(x)\le \frac{1}{\sqrt{\eta}^n}e^{\sqrt{\eta}x}
\end{equation}
and
\begin{equation}\label{eq_poly_n_2}
\forall x\ge 0,\qquad \frac{1}{2}\frac{x^n}{n!}\le \varphi_{[n]}(x)\le \sum_{k=0}^n\frac{1}{\sqrt{\eta}^{n-k}}\frac{x^k}{k!}.
\end{equation}
\end{Claim}

\begin{proof}
We argue by induction on $n$.\\
Note that $\varphi_{[0]}=\varphi$ is an increasing function and that 
\[ \forall t\ge 0,\quad\arctan{t}\le t. \]
Thus
\begin{align*}
\forall x\le 0,&\qquad 0\le \varphi_{[0]}(x)\le \frac{2}{\pi}e^{\sqrt{\eta}x}\le e^{\sqrt{\eta}x}\\
\forall x\ge 0,&\qquad \frac{1}{2}=\varphi(0)\le \varphi_{[0]}(x)\le 1.
\end{align*}
Now assume that \eqref{eq_poly_n_1} and \eqref{eq_poly_n_2} hold for some $n\in\NN$ being fixed. 

\bigskip

By definition of $\varphi_{[n+1]}$ and by the induction assumption, we have for all $x\le 0$
\[ 0 \le \varphi_{[n+1]}(x)\le \int_{-\infty}^x \frac{1}{\sqrt{\eta}^{n}}e^{\sqrt{\eta}t}\;dt\le \frac{1}{\sqrt{\eta}^{n+1}}e^{\sqrt{\eta}x}. \]
In particular, 
\[ 0\le\varphi_{[n+1]}(0)\le \frac{1}{\sqrt{\eta}^{n+1}}. \]
By the induction assumption, we then infer that for all $x\ge 0$,
\[ \varphi_{[n+1]}(x)=\varphi_{[n+1]}(0)+\int_0^x\varphi_{[n]}(t)\;dt \]satisfies
\[ 0+\int_0^x\frac{1}{2}\frac{t^n}{n!}\;dt\le \varphi_{[n+1]}(x)\le \frac{1}{\sqrt{\eta}^{n+1}}+\int_0^x\sum_{k=0}^n\frac{1}{\sqrt{\eta}^{n-k}}\frac{t^k}{k!}\;dt. \]
Thus for all $x\ge 0$,
\[ \frac{1}{2}\frac{x^{n+1}}{(n+1)!}\le \varphi_{[n+1]}(x)\le \frac{1}{\sqrt{\eta}^{n+1}}+\sum_{k=0}^{n}\frac{1}{\sqrt{\eta}^{n-k}}\frac{x^{k+1}}{(k+1)!}= \sum_{k=0}^{n+1}\frac{1}{\sqrt{\eta}^{n+1-k}}\frac{x^k}{k!}. \]
This finishes the induction argument, hence the proof of Claim \ref{claim_poly}.
\end{proof}

\begin{Claim}\label{claim_est_z}
For all $n\in\NN$, for all $s\in\NN$, there exists $K_{s,n}\ge 0$ such that for all $t\ge T_0$,
\[ \int_{x\ge \beta t}\left(\partial_x^sz\right)^2(x-\beta t)^n\;dx\le K_{s,n}e^{-\theta t}. \]
\end{Claim}

\begin{proof}
We have already observed that $\varphi_{[n]}$ belongs to $\mathcal{E}$. Using Claim \ref{claim_poly}, we deduce that for all $s\in\NN$, 
\begin{align*}
\int_{x\ge \beta t}\left(\partial_x^sz\right)^2(x-\beta t)^n\;dx\le 2n!\int_\RR\left(\partial_x^sz\right)^2\varphi_{[n]}(x-\beta t)\;dx \le K_{s,n}e^{-\theta t},
\end{align*}
where $K_{s,n}:=2n!C(s,\varphi_{[n]})$.
\end{proof}

Proposition \ref{prop_decroissance} follows now from Claim \ref{claim_est_z}, from the Sobolev embedding $H^1(\RR)\hookrightarrow L^\infty(\RR)$, and the decay of all derivatives of the $R_j$.
\end{proof}

\begin{Rq}\label{rq_surpoly}
For all $s\in\NN$, there exists $\varphi$ satisfying ($\mathbf{A}$) and ($\mathbf{B}$($s$)) and which grows faster than each polynomial function as $x\to +\infty$. This follows from Proposition \ref{procede_tri}: given $\tilde{\mu_0}>\max\left\{\sqrt{p},\frac{p+1}{2}\right\}$, the sum of the series of functions
\[ \sum_{n\ge 0}\frac{x^n}{2^{\tilde{\mu_0}^{s+n}}}  \in \mathcal E. \]
See paragraph \ref{subs_appendix_3} in the Appendix for details.
\end{Rq}

\section{Rapid decay of the (NLS) multi-solitons}

In this section, we turn to the case of the nonlinear Schrödinger equation \eqref{NLS} (without computing explicitly the dependence of the constants with respect to the differential parameter $s$) and, by adapting the technique exposed in the previous section, we prove Theorem \ref{th_decay_NLS} (which focuses on the region far away the solitons). We do it for general smooth nonlinearities $g$ as in Remark \ref{rq:generalization}.

 Concerning the exponential decay of (NLS) multi-solitons inside the solitons region \eqref{exp_decay_solitons_region}, the proof is completely similar to that of Proposition \ref{prop:soliton} and is left to the reader.

\bigskip

Let us take $\beta>\max_j\{|v_j|\}$ and $0<\eta\le\min_j\{\omega_j\}$. Note that with this choice of parameters $\beta$ and $\eta$, the following interaction property (analogous to \eqref{ineg_interaction}) holds:

\begin{Claim}\label{claim_interaction_NLS}
For all $s\in\NN^d$, there exists $C(s)>0$ such that for all $j=1,\dots,N$ and for all $t\ge T_0$,
\begin{equation}
\int_{\RR^d}\left|\partial^s R_j(t,x)\right|e^{\sqrt{\eta}(|x|-\beta t)}\;dx \le C(s).
\end{equation}
\end{Claim}

\begin{proof}
Let us use the exponential decay of the ground states \eqref{exp_decay_ground_state} and observe that
\begin{align*}
\int_{|x|\le\beta t}e^{-\sqrt{\omega_j}|x-v_jt-x_j^0|}e^{\sqrt{\eta}(|x|-\beta t)}\;dx&\le \int_{|x|\le\beta t}e^{-\sqrt{\omega_j}|x-v_jt-x_j^0|}\;dx \le \int_{\RR^d}e^{-\sqrt{\omega_j}|x|}\;dx\le C(j).
\end{align*}
On the other hand, we obtain
\begin{align*}
\int_{|x|>\beta t}e^{-\sqrt{\omega_j}|x-v_jt-x_j^0|}e^{\sqrt{\eta}(|x|-\beta t)}\;dx&\le Ce^{(\sqrt{\omega_j}|v_j|-\sqrt{\eta}\beta)t}\int_{|x|> \beta t}e^{-(\sqrt{\omega_j}-\sqrt{\eta})|x|}\;dx\\
&\le Ce^{(\sqrt{\omega_j}|v_j|-\sqrt{\eta}\beta)t}C_de^{-(\sqrt{\omega_j}-\sqrt{\eta})\beta t}\\
&\le Ce^{\sqrt{\omega_j}(|v_j|-\beta)t}.\qedhere
\end{align*}
\end{proof}

We recall the notation $z:=u-\sum_{j=1}^NR_j$ and the decay rate $\theta>0$ such that
\[ \forall t\ge T_0,\quad \|z(t)\|_{H^s}\le C_se^{-\theta t}. \]
In the spirit of the previous subsection, we consider weight functions $\varphi\in\mathscr{C}^1(\RR,\RR)$ such that the following assumptions are satisfied:

\begin{align}
&(i) \quad\displaystyle\lim_{x\to -\infty}\varphi(x)=0 \tag{$\mathbf{A'}$}\\
&(ii) \quad\exists \kappa_1>0,\:\forall x\in\RR,\quad 0\le \varphi '(x)\le \kappa_1e^{\sqrt{\eta} x}\notag
\end{align}

and for given $\overline{s}\in\NN$,

\begin{align}
&\exists C({\overline{s}},\varphi)>0,\:\forall \sigma\in\NN^d,\quad |\sigma|\le \overline{s}\Rightarrow\forall t\ge T_0,\quad \int_{\RR}\left|\partial^\sigma z(t,x)\right|^2\varphi(|x|-\beta t)\;dx\le C(\overline{s},\varphi)e^{-\theta t}.\tag{$\mathbf{B'}(\overline{s})$}
\end{align}

Similarly to the case of the (gKdV) equation, we define the following property which depends on $\overline{s}$ and $\varphi$:
\[ P'(\overline{s},\varphi):\quad \varphi \text{ satisfies } (\mathbf{A'}) \text{ and } (\mathbf{B'}(\overline{s})). \]
We state

\begin{Prop}\label{prop_P'_NLS}
For any $\overline{s}\in\NN$, $P'(\overline{s}+1,\varphi)\Rightarrow P'(\overline{s},\varphi_{[1]})$.
\end{Prop}

\begin{proof}

The proof follows the same scheme as that of Proposition \ref{procede_tri}.
We introduce the family of integrals

\[ J_{s,x_0}(t):=\sum_{\sigma\in\NN^d,\;|\sigma|=s}\int_{\RR^d}\left|\partial^\sigma z\right|^2(t,x)\varphi(|x|-x_0-\beta t)\;dx, \]
for all $s\in\NN$ and $x_0>0$.

We then show the following induction formula which makes the link between the functions $J_{s,x_0}$, $s\in\NN$. Here again we denote $\tilde{x}=\tilde{x}(t):=|x|-x_0-\beta t$ (for $x_0>0$ and $t\ge T_0$). 

\begin{Lem}\label{lem_est_rec_J}
For all $s\in\NN$, there exists $C_s\ge 0$ (independent of $x_0$) such that for all $t\ge T_0$:
\begin{equation}\label{est_rec_J}
\left|\frac{d}{dt}J_{s,x_0}(t)\right|\le C_s\int_{\RR^d}\sum_{|\sigma |\le s+1}\left|\partial^{\sigma}z\right|^2(t,x)\varphi'(\tilde{x})\;dx +C_se^{-\theta t}\sum_{k=0}^{s}J_{k,x_0}(t)+C_se^{-\sqrt{\eta}x_0}e^{-\theta t}.
\end{equation}
\end{Lem}

\begin{Rq}
Notice that estimate \eqref{est_rec_J} of $\frac{d}{dt}J_{s,x_0}$ obtained in Lemma \ref{lem_est_rec_J} contains one additional term (namely $e^{-\theta t}J_{s,x_0}(t)$) with respect to the corresponding estimate \eqref{est_der_I_1} obtained in the case of the \eqref{gKdV} equation. This is due to the algebra linked with the structure of the \eqref{NLS} equation; for the gKdV case, we manage to eliminate this term by means of one integration by parts.
\end{Rq}

\begin{proof}
Let us take $s\in\NN$ and $\sigma\in\NN^d$ such that $|\sigma|=s$. Denoting
$e_l$ the $d$-tuple $(0,\dots,1,\dots,0)$ for which all components except the $l$-th one are zero, we compute:
\begin{align*}
\frac{d}{dt}\int_{\RR^d}|\partial^\sigma z(t,x)|^2\varphi(|x|-x_0-\beta t)\;dx=&2\R\int_{\RR^d}\partial_t\overline{\partial^\sigma z}\varphi(\tilde{x})\;dx-\beta\int_{\RR^d}|\partial^\sigma z(t,x)|^2\varphi '(\tilde{x})\;dx \\
=&-2\I\int_{\RR^d}\sum_{l=1}^d \partial^{\sigma+2e_l}z\overline{\partial^\sigma z}\varphi(\tilde{x})\;dx-\beta\int_{\RR^d}|\partial^\sigma z(t,x)|^2\varphi '(\tilde{x})\;dx \\
&-2\I\int_{\RR^d}\partial^\sigma\left(g(z+R)-\sum_{j=1}^Ng(R_j)\right)\overline{\partial^\sigma z}\varphi(\tilde{x})\;dx\\
=&2\I\sum_{l=1}^d\int_{\RR^d}\partial^{\sigma+e_l}z\overline{\partial^\sigma z}\frac{x_l}{|x|}\varphi '(\tilde{x})\;dx-\beta\int_{\RR^d}|\partial^\sigma z(t,x)|^2\varphi '(\tilde{x})\;dx\\
&-2\I\int_{\RR^d}\partial^\sigma\left(g(z+R)-\sum_{j=1}^Ng(R_j)\right)\overline{\partial^\sigma z}\varphi(\tilde{x})\;dx.
\end{align*}
Notice that the last line results from one integration by parts.

\bigskip

Let us explain how to estimate the integrals appearing in the last equality. For all $l=1,\dots,d$,
\begin{equation}\label{est_aux_1}
\left|\I\int_{\RR^d}\partial^{\sigma+e_l}z\overline{\partial^\sigma z}\frac{x_l}{|x|}\varphi '(\tilde{x})\;dx\right|\le \frac{1}{2}\int_{\RR^d}|\partial^{\sigma+e_l} z(t,x)|^2\varphi '(\tilde{x})\;dx +\frac{1}{2}\int_{\RR^d}|\partial^\sigma z(t,x)|^2\varphi '(\tilde{x})\;dx.
\end{equation}

For the last term, let us decompose $g(z+R)-\sum_{j=1}^Ng(R_j)$ as follows:
\[ g(z+R)-\sum_{j=1}^Ng(R_j)=\left[g(z+R)-g(z)-g(R)\right]+[g(z)]+\left[g(R)-\sum_{j=1}^Ng(R_j)\right]. \]
The control of 
\[ \int_{\RR^d}\partial^\sigma\left(g(R)-\sum_{j=1}^Ng(R_j)\right)\overline{\partial^\sigma z}\varphi(\tilde{x})\;dx \]
uses Claim \ref{claim_interaction_NLS} (we refer to the similar control of $I_1$ in the proof of Lemma \ref{lem_der_I}). In this way, we obtain
\begin{equation}\label{est_aux_2}
\left|\int_{\RR^d}\partial^\sigma\left(g(R)-\sum_{j=1}^Ng(R_j)\right)\overline{\partial^\sigma z}\varphi(\tilde{x})\;dx\right|\le Ce^{-\sqrt{\eta}x_0}e^{-\theta t}. 
\end{equation}

By the Faà di Bruno formula, $\partial^\sigma(g(z))$ rewrites as a linear combination of the following terms
\[ \partial^{\sigma_1}z_{j_1}\dots\partial^{\sigma_q}z_{j_q}\frac{\partial^q g}{\partial_x^{r}\partial_y^{q-r}}(z), \]
where $q\in\{0,\dots,|\sigma|\}$, $r\in\{0,\dots,q\}$, $\sum_{m=1}^q|\sigma_m|=|\sigma|$, $j_l\in\{1,2\}$ for $m=1,\dots,q$, and $z_1:=\R (z)$, $z_2:=\I (z)$.

We observe that $\partial^{\sigma_l}z_{j_m}\in L^\infty(\RR^d)$ since $z(t)\in H^\infty(\RR^d)$ and due to the Sobolev embedding $H^s(\RR^d)\hookrightarrow L^\infty(\RR^d)$ available for all $s>\frac{d}{2}$. Now, by means of
\[ \left|\frac{\partial^q g}{\partial_x^{r}\partial_y^{q-r}}(z)\right|\le C|z|^{p-q} \quad\text{if } q\le p \]
and 
\[ \left|\frac{\partial^q g}{\partial_x^{r}\partial_y^{q-r}}(z)\right|\le C \quad\text{if } q> p, \]
one can bound
\begin{equation}\label{est_aux_3}
\left|\int_{\RR^d}\partial^{\sigma}(g(z))\overline{\partial^\sigma(z)}\varphi(\tilde{x})\;dx\right|\le Ce^{-\theta t}\sum_{|\sigma '|\le s}\int_{\RR^d}|\partial^{\sigma '}z|^2\varphi(\tilde{x})\;dx.
\end{equation}

We finally deal with the integral $\int_{\RR^d}\partial^\sigma\left(g(z+R)-g(z)-g(R)\right)\overline{\partial^\sigma(z)}\varphi(\tilde{x})\;dx$.
By the Faà di Bruno formula applied to $\partial^\sigma\left(g(z+R)-g(z)-g(R)\right)$, it suffices to consider each quantity of the form 
\[ \partial^{\sigma_1}(z+R)_{j_1}\dots\partial^{\sigma_q}(z+R)_{j_q}\frac{\partial^q g}{\partial_x^{r}\partial_y^{q-r}}(z+R)-\partial^{\sigma_1}z_{j_1}\dots\partial^{\sigma_q}z_{j_q}\frac{\partial^q g}{\partial_x^{r}\partial_y^{q-r}}(z)-\partial^{\sigma_1}R_{j_1}\dots\partial^{\sigma_q}R_{j_q}\frac{\partial^q g}{\partial_x^{r}\partial_y^{q-r}}(R). \]
(We keep the same notations for indices and differential parameters as above.) 
Now the desired estimation is based on the decomposition
\begin{multline*}
\left[\partial^{\sigma_1}(z+R)_{j_1}\dots\partial^{\sigma_q}(z+R)_{j_q}-\partial^{\sigma_1}z_{j_1}\dots\partial^{\sigma_q}z_{j_q}-\partial^{\sigma_1}R_{j_1}\dots\partial^{\sigma_q}R_{j_q}\right]\frac{\partial^q g}{\partial_x^{r}\partial_y^{q-r}}(z+R)\\
+\partial^{\sigma_1}z_{j_1}\dots\partial^{\sigma_q}z_{j_q}\left[\frac{\partial^q g}{\partial_x^{r}\partial_y^{q-r}}(z+R)-\frac{\partial^q g}{\partial_x^{r}\partial_y^{q-r}}(z)\right]+\partial^{\sigma_1}R_{j_1}\dots\partial^{\sigma_q}R_{j_q}\left[\frac{\partial^q g}{\partial_x^{r}\partial_y^{q-r}}(z+R)-\frac{\partial^q g}{\partial_x^{r}\partial_y^{q-r}}(R)\right]
\end{multline*}
and on Claim \ref{claim_interaction_NLS}.
We have
\begin{equation}\label{est_aux_4}
\left|\int_{\RR^d}\partial^\sigma\left(g(z+R)-g(z)-g(R)\right)\overline{\partial^\sigma(z)}\varphi(\tilde{x})\;dx\right|\le Ce^{-2\theta t}e^{-\sqrt{\eta}x_0}.
\end{equation}

We obtain the Lemma by gathering the previous estimates \eqref{est_aux_1}, \eqref{est_aux_2}, \eqref{est_aux_3}, and \eqref{est_aux_4}.
 
\end{proof}

Let us now go on with the proof of Proposition \ref{prop_P'_NLS}. By integration of estimate
\begin{equation}\label{est_rec_J_bis}
\left|\frac{d}{dt}J_{s,x_0}(t)\right|\le C_s\int_{\RR^d}\sum_{|\sigma |\le s+1}\left|\partial^{\sigma}z\right|^2(t,x)\varphi'(\tilde{x})\;dx +C_se^{-\theta t}\sum_{k=0}^{s}J_{k,x_0}(t)+C_se^{-\sqrt{\eta}x_0}e^{-\theta t}
\end{equation}
 (which directly follows from \eqref{est_rec_J} and the definition of $J_{s,x_0}$) between $t$ and $+\infty$, the term $e^{-\theta t}\sup_{t'\ge t}J_{s,x_0}(t')$ can be absorbed for large values of $t$ and we are lead to a similar result to that of Lemma \ref{lem_est_I}: for all $t$ sufficiently large,
\begin{equation}
J_{s,x_0}(t)\le K_s\int_{t}^{+\infty}\int_{\RR}\left(\sum_{k=0}^{s+1}\left|\partial_x^{k}z\right|^2(t',x)\varphi'(\tilde{x}(t'))\right)\;dx\;dt'+K_se^{-\sqrt{\eta}x_0}e^{-\theta t}.
\end{equation}
for some constant $K_s$ independent of $x_0$. 

By this means, we can complete the proof of Proposition \ref{prop_P'_NLS}, and then the algebraic decay of $z$, as it was done for the \eqref{gKdV} multi-solitons in Section \ref{subsec_rapid_dec}.

\end{proof}

\appendix
\section{Appendix}

\subsection{Growth of the \texorpdfstring{$H^s$}{Hs} norms of the 1-D solitons}

The purpose of this appendix is to make the constants more explicit as stated in Section \ref{sec_decay_gkdv_right}. We restrict to the space dimension 1 and monomial nonlinearity.

\begin{Prop}\label{prop_est_Hs_solitons}
For all $\mu>\sqrt{p}$, there exists $s_0$ such that for all $s\ge s_0$,
\begin{equation}
\|Q\|_{H^s}\le 2^{\mu^s}.
\end{equation}
 \end{Prop}

\begin{proof}
Differentiating the fundamental equation satisfied by $Q$, that is $Q''+Q^p=Q$, we obtain the following induction formula:
\[ \forall s\in\NN,\quad Q^{(s+2)}=Q^{(s)}-\sum_{i_1+\dots+i_p=s}\dbinom{s}{i_1,\dots,i_p}Q^{(i_1)}\dots Q^{(i_p)}. \]
Let us observe that for $i_1,\dots,i_p\in\NN$ such that $i_1+\dots +i_p=s$,
\begin{itemize}
\item
if there exists $j\in\{1,\dots,p\}$ such that $i_j=s$, then
\begin{align*}
 \int_\RR\left(Q^{(i_1)}\dots Q^{(i_p)}\right)^2\;dx & =\int_\RR\left(Q^{(s)}\right)^2 Q^{2(p-1)}\;dx \\
&  \le \|Q\|_{L^\infty}^{2(p-1)}\int_\RR\left(Q^{(s)}\right)^2 \;dx \le C\|Q\|_{H^1}^{2(p-1)}\|Q\|_{H^s}^2
\end{align*}
($C$ being a constant depending only on $p$);
\item if for all $j\in\{1,\dots,p\}$ such that $i_j\le s-1$, then
\begin{align*}
 \int_\RR\left(Q^{(i_1)}\dots Q^{(i_p)}\right)^2\;dx & \le \prod_{k=2}^p\|Q^{(i_k)}\|^2_{L^\infty}\int_\RR\left(Q^{(i_1)}\right)^2 \;dx\\
&\le C\prod_{k=2}^p\|Q\|_{H^{i_{k}+1}}^{2}\|Q\|_{H^{i_1}}^2 \le C\|Q\|_{H^s}^{2(p-1)}\|Q\|_{H^{s-1}}^2,
\end{align*}
($C$ being again a constant depending only on $p$, which can change from one line to the other).
\end{itemize}

Thus for $s\in\NN^*$, 
\[ \|Q^{(s+2)}\|_{L^2}\le C\left(\|Q^{(s)}\|_{L^2}+p^s\|Q\|_{H^s}^p\right), \]
which implies 
\[ \|Q\|_{H^{s+2}}\le Cp^s\|Q\|_{H^s}^p, \]
 with a constant $C$ depending only on $p$.

We finally obtain
\begin{align*}
\|Q\|_{H^s}&\le C^{\frac{p^{\lfloor\frac{s}{2}\rfloor}-1}{p-1}}p^{p^{\lfloor\frac{s}{2}\rfloor} s}\|Q\|_{H^1} \le 2^{\mu^s},
\end{align*}
with $\mu>\sqrt{p}$ and $s$ sufficiently large.
\end{proof}

\begin{Rq}
As a corollary of Proposition \ref{prop_est_Hs_solitons}, we also obtain the existence of a constant $C$ depending on $p$ and the soliton parameters such that for all $s\in\NN$,
\begin{equation}
\left\|\left(\sum_{j=1}^NR_j(t)\right)^p-\sum_{j=1}^NR_j(t)^p\right\|_{H^s}\le Cp^s\max_{j=1,\dots,N}\|R_j\|_{H^s}^pe^{-2\theta t}.
\end{equation}
\end{Rq}

\subsection{Proof of estimate \texorpdfstring{\eqref{est_lambda_s}}{(3.3)}} \label{app:a3}

The goal is here to make explicit the constants appearing in the computations done by Martel \cite[Section 3.4]{martel} in the proof of smoothness of the multi-solitons. 

We thus repeat the arguments developed by Martel (presented slightly differently), keeping track of the growth of the constant $\lambda_s$ (with respect to $s$) such that 
\[ \forall t\ge T_0, \quad \|z(t)\|_{H^s}\le \lambda_se^{-\theta t}. \]

\begin{proof}
We consider regularity indices $s\ge 5$, as that case makes the argument easier: the point being that the exponential decay rate $\theta$ does not change when we go from the estimation of $\|z\|_{H^{s-1}}$ to that of $\|z\|_{H^s}$; there is a loss for $s=2$ (treated in detail in \cite{martel}), which can be avoided for $s=3,4$ using an extra argument, see the footnote below in the proof.

The starting point is to study the variations of
\[ \frac{d}{dt}\int_{\RR}\left(\partial_x^{s}z\right)^2\;dx=2\int_\RR\partial_x^s\left(\left(z+R\right)^p-\sum_{j=1}^NR_j^p\right)\partial_x^{s+1}z\;dx. \]

Thus, the terms which have to be controlled are the source term (involving $z$ only linearly)
\[ 2\int_{\RR}\partial_x^{s+1}\left(R^p-\sum_{j=1}^NR_j^p\right)\partial_x^sz\;dx. \]
and
\begin{gather*}
-2\int_{\RR}\partial_x^{s+1}\left(\sum_{k=1}^p\dbinom{p}{k}z^kR^{p-k}\right)\partial_x^sz\;dx= -2\sum_{k=1}^p\dbinom{p}{k}\sum_{i_1+\dots+i_p=s+1}\dbinom{s+1}{i_1,\dots,i_p}I_{k,i_1,\dots,i_p},
\\
\text{where} \quad I_{k,i_1,\dots,i_p}=\int_\RR\partial_x^{i_1}z\dots \partial_x^{i_k}z\partial_x^{i_{k+1}}R\dots\partial_x^{i_p}R\partial_x^sz\;dx. 
\end{gather*}
\emph{(i)} For $k=1$, integrating by parts, we have
\begin{align}
\MoveEqLeft
\sum_{i_1+\dots+i_p=s+1}\dbinom{s+1}{i_1,\dots,i_p}\int_\RR\partial_x^{i_1}z\partial_x^{i_2}R\dots\partial_x^{i_p}R\partial_x^sz\;dx \nonumber \\
& = \int_\RR\partial_x^{s+1}zR^{p-1}\partial_x^sz\;dx+(s+1)(p-1)\int_\RR \partial_x^sz\partial_xRR^{p-2}\partial_x^sz\;dx \nonumber \\
& \qquad +\sum_{\underset{\forall j,i_j\le s-1}{i_1+\dots+i_p=s+1}}\int_\RR\partial_x^{i_1}z\partial_x^{i_2}R\dots\partial_x^{i_p}R\partial_x^sz\;dx \nonumber \\
& = \frac{2s+1}{2}\int_\RR\left(\partial_x^sz\right)^2\partial_x(R^{p-1})\;dx+\sum_{\underset{\forall j,i_j\le s-1}{i_1+\dots+i_p=s+1}}\int_\RR\partial_x^{i_1}z\partial_x^{i_2}R\dots\partial_x^{i_p}R\partial_x^sz\;dx. \label{term1}
\end{align}
The first integral
\begin{equation} \label{term2}
\frac{2s+1}{2} \int_\RR\left(\partial_x^sz\right)^2\partial_x(R^{p-1})\;dx
\end{equation}
can not be bounded directly in a suitable way. The key idea in \cite{martel} is to add a lower order term
\[ \frac{2s+1}{3}p\int_{\RR}\left(\partial_x^{s-1}z\right)^2R^{p-1}\;dx \]
whose variation at leading order will precisely cancel \eqref{term2}. Thus we are lead to consider
\[ F_s(t):=\int_{\RR}\left(\partial_x^{s}z\right)^2\;dx-\frac{2s+1}{3}p\int_{\RR}\left(\partial_x^{s-1}z\right)^2R^{p-1}\;dx. \]
Going back to \eqref{term1}, we bound
\begin{equation}
\label{est:term1}
\left|\sum_{\underset{\forall j,i_j\le s-1}{i_1+\dots+i_p=s+1}}\int_\RR\partial_x^{i_1}z\partial_x^{i_2}R\dots\partial_x^{i_p}R\partial_x^sz\;dx\right|\le  Cp^{s+1}\|z\|_{H^{s-1}}^2\|R\|_{H^s}^{p-1}.
\end{equation}

\emph{(ii)}  Let us now consider the case where $2\le k\le p$. If there exists $j\in\left\{1,\dots,k\right\}$ such that $i_j=s+1$, then integrating by parts,
\begin{align}
| I_{k,i_1,\dots,i_p}| & = \frac{1}{2} \left| \int_\RR\left(\partial_x^sz\right)^2\partial_x\left(z^{k-1}R^{p-k}\right)\;dx \right| \nonumber \\
& \le \frac{1}{2}\left\|\partial_x\left(R^{p-k}z^{k-1}\right)\right\|_{L^\infty}\|z\|_{H^s}^2 \le \frac{1}{2}\|R\|_{H^2}^{p-k}\|z\|_{H^2}^{k-1}\|z\|_{H^s}^2. \label{est:term2}
\end{align}
If there exists $j\in\left\{1,\dots,k\right\}$ such that $i_j=s$, then 
\begin{equation}
|I_{k,i_1,\dots,i_p}|\le C\|z\|_{H^2}^{k-1}\|R\|_{H^2}^{p-k}\int_{\RR}\left(\partial_x^sz\right)^2\;dx. \label{est:term3}
\end{equation}

If there exists $j\in\left\{1,\dots,k\right\}$ such that $i_j=s-1$, then for all $j'\in\{1,\dots,k\}$ such that $j'\ne j$, $i_{j'}\le \sum_{l=1}^pi_l-s+1\le 2\le s-3$ (by the choice of $s\ge 5$)\footnote{If $s=3$ or $4$, a term which is cubic in $\partial_x^{s-1} z$, of the type
\[ \int (\partial_x^{s-1} z)^3 P(R_j, z, \partial_x z)\; dx \]
can occur, where $P$ is some function (but there are no terms with higher power of $\partial_x^{s-1} z$). Via the Gagliardo-Nirenberg inequality, it can be bounded by
\[ \left( \| z \|_{H^s}^{1/6} \| z \|_{H^{s-1}}^{5/6} \right)^3 \| P(R_j,z,\partial_x z) \|_{L^\infty} \lesssim C(1+\| z \|_{H^2})^{p-2} e^{-5\theta/2 t}  \| z \|_{H^s}^{1/2},   \]
the point being that the decay rate in $\theta$ is greater than $2$: one can then complete the estimates as written here for $s \ge5$.
}. Hence, integrating by parts,

\begin{align}
|I_{k,i_1,\dots,i_p}| &= \left|-\frac{1}{2}\int_\RR\left(\partial_x^{s-1}z\right)^2\partial_x\left(\partial_x^{i'_1}z\dots\partial_x^{i'_{k-1}}z\partial_x^{i_{k+1}}R\dots\partial_x^{i_p}R\right)\;dx \right| \nonumber \\
&\le \frac{1}{2}\|z\|_{H^{s-1}}^2\left\|\partial_x\left(\partial_x^{i'_1}z\dots\partial_x^{i'_{k-1}}z\partial_x^{i_{k+1}}R\dots\partial_x^{i_p}R\right)\right\|_{L^\infty} \nonumber \\
& \le C\|z\|_{H^{s-1}}^{k+1}\|R\|_{H^{s+1}}^{p-k}, \label{est_appendix}
\end{align}
where $C$ is a universal constant depending only on $p$.

In the other cases, $i_1,\dots,i_{k}\le s-2$ and so
\begin{align}
|I_{k,i_1,\dots,i_p}|&\le\; \left|-\int_\RR\partial_x^{s-1}z\partial_x\left(\partial_x^{i_1}z\dots\partial_x^{i_k}z\partial_x^{i_{k+1}}R\dots\partial_x^{i_p}R\right)\;dx\right| \le  \|z\|_{H^{s-1}}^{k+1}\|R\|_{H^{s-1}}^{p-k}. \label{est:term4}
\end{align}

\emph{(iii)} For the source term, by integration by parts,
\begin{align*}
\left|2\int_{\RR}\partial_x^{s+1}\left(R^p-\sum_{j=1}^NR_j^p\right)\partial_x^sz\;dx\right|
&\le C\|z\|_{H^{s-1}}\left\|R^p-\sum_{j=1}^NR_j^p\right\|_{H^{s+2}}\\
&\le C\lambda_{s-1}e^{-\theta t}\left\|R^p-\sum_{j=1}^NR_j^p\right\|_{H^{s+2}}.
\end{align*}

\emph{(iv)} When computing the time differential of $\displaystyle \int_{\RR}\left(\partial_x^{s-1}z\right)^2R^{p-1}\;dx$, as mentioned (and by construction), one term cancels \eqref{term2} and the others are bounded as in $(i)$ and $(ii)$.

\bigskip

Gathering the bounds \eqref{est:term1},  \eqref{est:term2},  \eqref{est:term3}, \eqref{est_appendix},  \eqref{est:term4}, it results that
\begin{align}
\left|\frac{d}{dt}F_s(t)\right|\le & \; C\lambda_{s-1}e^{-\theta t}\left\|R^p-\sum_{j=1}^NR_j^p\right\|_{H^{s+1}} +C\sum_{k=2}^p\dbinom{p}{k}\|z\|_{H^s}^2p^{s+1}\|z\|_{H^2}^{k-1}\|R\|_{H^s}^{p-k} \nonumber \\
& \quad +C\sum_{k=2}^p\dbinom{p}{k}p^{s+1}\|z\|_{H^{s-1}}^{k+1}\|R\|_{H^s}^{p-k} \nonumber \\
& \le C\lambda_{s-1}e^{-\theta t}p^{s+2}2^{\mu^{s+2}p}e^{-2\theta t}+C\|z\|_{H^s}^2e^{-2\theta t}p^{s+1}2^p2^{\mu^s(p-2)} \nonumber  \\
& \quad +Ce^{-3\theta t}p^{s+1}2^p2^{\mu^s(p-2)}\lambda_{s-1}^{p+1} \nonumber \\
& \le  C\left(p^s2^{\mu^sp}e^{-2\theta t}\|z\|_{H^s}^2+p^s2^{\mu^sp}\lambda_{s-1}^{p+1}e^{-3\theta t}\right). \label{est_Fs}
\end{align}

Moreover, we have by definition of $F_s$,
\begin{align*}
\int_{\RR}\left(\partial_x^{s}z\right)^2\;dx&\le |F_s|+Cs\|R\|_{H^2}^{p-1}\|\partial_x^{s-1}z\|_{L^2}^2 \le  |F_s|+Cs\lambda_{s-1}^2e^{-2\theta t}.
\end{align*}

Hence, 
\[ \left|\frac{d}{dt}F_s(t)\right|\le Cp^s2^{\mu^sp}e^{-2\theta t}\left(|F_s(t)|+s\lambda_{s-1}^2e^{-2\theta t}\right)+Cp^s2^{\mu^sp}\lambda_{s-1}^{p+1}e^{-3\theta t}. \]

By integration we finally obtain
\[ |F_s(t)|\le Cp^{s}2^{\mu^sp}\lambda_{s-1}^{p+1}e^{-2\theta t} \]
and thus
\[ \|z\|_{H^{s}}^2\le Cp^{s}2^{\mu^sp}\lambda_{s-1}^{p+1}e^{-2\theta t}. \]
We can therefore take
\begin{equation}
\lambda_{s}\le Cp^{\frac{s}{2}}2^{\frac{p}{2}\mu^s}\lambda_{s-1}^{\frac{p+1}{2}},
\end{equation}
which yields via an induction
\[ \lambda_{s}\le C^{\sum_{k=0}^{s-1}\left(\frac{p+1}{2}\right)^k}p^{\sum_{k=0}^{s-1}\frac{s-k}{2}\left(\frac{p+1}{2}\right)^k}\lambda_0^{\left(\frac{p+1}{2}\right)^{s}}2^{\frac{p}{2}\sum_{k=0}^{s-1}\mu^{s-k}\left(\frac{p+1}{2}\right)^k}, \]
and so, 
\begin{align} \label{est_appendix_lambda}
\lambda_{s}&\le C^{2\left(\frac{p+1}{2}\right)^{s}}p^{\frac{s}{2}\left(\frac{p+1}{2}\right)^{s-1}}2^{\left(\max\{\mu,\frac{p+1}{2}\}\right)^s}\lambda_0^{\left(\frac{p+1}{2}\right)^{s}} \le 2^{\mu_0^s},
\end{align}
with $\mu_0>\max\{\mu,\frac{p+1}{2}\}>\max\left\{\sqrt{p},\frac{p+1}{2}\right\}$ and $s$ sufficiently large (depending on $\mu_0$). 
\end{proof}

\begin{Rq}
Note that we can refine a bit \eqref{est_appendix} but the final estimate concerning $\lambda_s$ would not be better than \eqref{est_appendix_lambda}, considering that $\mu_0$ is to be chosen strictly greater than $\max\{\sqrt{p},\frac{p+1}{2}\}$.
\end{Rq}

\subsection{Details concerning Remarks \ref{rq_exp} and \ref{rq_surpoly}}\label{subs_appendix_3}

Let $s\in\NN$. In order to ensure that $\int_{x>\beta t}\left(\partial_x^sz\right)^2e^{\epsilon (x-\beta t)}\;dx$ is finite for some $\epsilon>0$, it suffices by \eqref{eq_poly_n_1} that
\[ \sum_{n=0}^{+\infty}\int_\RR \left(\partial_x^sz\right)^2\epsilon^n\varphi_{[n]}\;dx \]
 is finite. Thus it suffices that the series $\sum_{n\ge 0}C(s,\varphi_{[n]})\epsilon^n$ converges. 

This condition is satisfied under the following assumptions:
\begin{equation}\label{cond_dec_exp}
C(s,\varphi_{[1]}) \le c_0C(s+1,\varphi) \quad \text{and} \quad  \lambda_s \le \tilde{c_0}^s.
\end{equation}
Indeed, if we assume \eqref{cond_dec_exp}, we obtain
\[ C(s,\varphi_{[n]})\le c_0^n\tilde{c_0}^{2(s+n)}\le \tilde{c_0}^{2s}(c_0\tilde{c_0}^2)^n, \]
which guarantees the existence of $\epsilon>0$ such that the series $\sum_{n\ge 0}C(s,\varphi_{[n]})\epsilon^n$ converges.

\bigskip

From Proposition \ref{procede_tri}, we also deduce, proceeding step by step, that:
\begin{align*}
C(s,\varphi_{[n]})&\le c(\eta,\kappa_1,\kappa_2)2^{\mu_0^s}C(s+1,\varphi)\\
&\le  c(\eta,\kappa_1,\kappa_2)^n2^{\mu_0^s+\dots+\mu_0^{s+n-1}}C(s+n,\varphi)\\
&\le  c(\eta,\kappa_1,\kappa_2)^n2^{n\mu_0^{s+n-1}}\lambda_{s+n}^2.
\end{align*}
Taking $\tilde{\mu_0}>\mu_0$, this shows that the integral
\[ \int_{x \ge \beta t}\left(\partial_x^sz\right)^2(t,x)\left(\sum_{n=0}^{+\infty}\frac{(x-\beta t)^n}{2^{\tilde{\mu_0}^{s+n}}}\right)\;dx \]
is finite.

\bibliographystyle{plain}
\bibliography{references_dec_gkdv}
\nocite{*}

\end{document}